\documentclass[10pt,letterpaper,reqno]{amsart}

\usepackage[T1]{fontenc}
\usepackage[utf8]{inputenc}
\usepackage{lmodern}
\usepackage{amsmath,amssymb,mathtools}
\usepackage{xcolor}
\usepackage{hyperref}

\hypersetup{
	colorlinks=true,
	linkcolor=blue,
	citecolor=blue,
	urlcolor=blue,
	pdfauthor={Yufan Luo, Yiqi Xu},
	pdftitle={On the bounded-conductor finiteness conjecture in equal characteristic},
	pdfsubject={Equal-characteristic etale local systems with unbounded finite monodromy}
}

\newtheorem{theorem}{Theorem}[section]
\newtheorem{proposition}[theorem]{Proposition}
\newtheorem{lemma}[theorem]{Lemma}
\newtheorem{corollary}[theorem]{Corollary}
\newtheorem{conjecture}[theorem]{Conjecture}

\theoremstyle{definition}
\newtheorem{definition}[theorem]{Definition}
\theoremstyle{remark}
\newtheorem{remark}[theorem]{Remark}

\numberwithin{equation}{section}

\newcommand{\Fp}{\mathbf F_p}
\newcommand{\Fbar}{\overline{\mathbf F}_p}

\newcommand{\SL}{\operatorname{SL}}
\newcommand{\PSL}{\operatorname{PSL}}
\newcommand{\GL}{\operatorname{GL}}
\newcommand{\Gal}{\operatorname{Gal}}
\newcommand{\Ram}{\operatorname{Ram}}
\newcommand{\Frob}{\operatorname{Frob}_p}
\newcommand{\et}{\mathrm{\acute{e}t}}

\title[On the bounded-conductor finiteness conjecture]
{On the bounded-conductor finiteness conjecture in equal characteristic}

\author{Yufan Luo}
\address{Shanghai Institute for Mathematics and Interdisciplinary Sciences (SIMIS),
	Shanghai 200433, China}
\address{Research Institute of Intelligent Complex Systems, Fudan University,
	Shanghai 200433, China}
\email{yufanluo@hotmail.com}

\author{Yiqi Xu}
\address{Fakult{\"a}t 8 -- Mathematik und Physik,
	Universit{\"a}t Stuttgart,
	Stuttgart, Germany}
\email{xuyiqi1@gmail.com}

\date{\today}

\begin{document}
	
	\begin{abstract}
		We investigate the equal-characteristic case of the Moon--Taguchi bounded-conductor finiteness conjecture for mod $p$ Galois representations over global function fields. We first establish a conditional finiteness theorem: for any global function field $K$ of characteristic $p$ and any integer $n \ge 1$, there are only finitely many isomorphism classes of continuous, semisimple, everywhere unramified, and geometric representations $\rho: G_K \to \mathrm{GL}_n(\overline{\mathbf{F}}_p)$ that admit an everywhere unramified characteristic-zero lift. Furthermore, we prove that the conjecture fails in general without this lifting hypothesis. Concretely, we construct infinitely many global function fields $K$ of characteristic $p$ admitting a continuous, surjective, absolutely irreducible, everywhere unramified, and geometric representation $\rho_r: G_K \twoheadrightarrow \mathrm{SL}_2(\mathbf{F}_{p^r})$ for each integer $r \ge 4$. As a corollary, we construct an everywhere unramified Galois extension \(L/K\), regular over \(\mathbf F_p\), with Galois group
		\[
		\mathrm{Gal}(L/K)\cong
		\prod_{r\geq4}\mathrm{PSL}_2(\mathbf F_{p^r}).
		\]
		Combined with known cross-characteristic finiteness theorems, this completely resolves the question posed by Moon and Taguchi in dimension two: such extensions exist over global function fields of characteristic \(p\), whereas they cannot exist over global function fields of characteristic different from \(p\).
	\end{abstract}
	
	\maketitle
	
	\tableofcontents
	
	\section{Introduction}
	\subsection{Bounded-conductor finiteness conjecture}
	Let \(p\) be a prime number and let \(\Fp\) denote the
	finite field of order \(p\). Fix an algebraic closure \(\Fbar\) of
	\(\Fp\), and endow it with the discrete topology. Motivated by Serre's modularity conjecture \cite[Sect.~3.2, pp.~195--196]{MR885783} and the Fontaine--Mazur finiteness conjectures for \(p\)-adic Galois representations (see \cite[Sect.~3, Conj.~2a, p.~195]{FontaineMazur97}), Khare \cite[Conj.~2.2]{MR1751924} and Moon \cite[Sect.~4, Prob., p.~163]{MR1782427} independently proposed the following finiteness conjecture for mod \(p\) Galois representations (see also \cite[Sect.~7, p.~497]{Taguchi17}):
	
	\begin{conjecture}[Bounded-conductor finiteness conjecture in the number field case]
		Let \(K\) be a number field, \(n\) be a positive integer and 	\(\mathfrak N\) be a nonzero ideal of the ring of integers of \(K\). Let \(G_{K}\) be the absolute Galois group of \(K\). Then there are only finitely many isomorphism classes of continuous semisimple representations 
		\[
		\rho:G_K\longrightarrow\GL_n(\Fbar)
		\]
		such that the prime-to-\(p\) Artin conductor \(\mathfrak{R}(\rho)\) of \(\rho\) is bounded by \(\mathfrak N\). 
	\end{conjecture}
	
	For \(K=\mathbf{Q}\) and \(n=2\), the conjecture for odd
	representations follows from Serre's modularity conjecture,
	proved by Khare and Wintenberger
	\cite[Thm.~9.1, p.~500]{MR2551763}, \cite[Sect.~10, pp.~581--584]{MR2551764},
	using Kisin's modularity lifting theorem
	\cite[Thm.~0.1, pp.~587--588]{Kisin09}.
	For the deduction of finiteness, see
	\cite[Sect.~4, Rem.~(4), p.~163]{MR1782427}.
	More generally, for totally real number fields $K$, the
	conjecture for two-dimensional totally odd representations follows from either of the
	mod \(p\) local--global compatibility conjectures of Buzzard, Diamond and Jarvis
	\cite[Conjs.~4.7 and~4.9, pp.~149--150; Cor.~4.14, p.~156]{MR2730374}. An analogue of this finiteness conjecture in the function field setting
	was subsequently formulated by Moon and Taguchi
	\cite[Sect.~4, pp.~2533--2534]{MoonTaguchi01}; see also the related moduli-theoretic discussion in \cite[Sect.~7, pp.~497--499]{Taguchi17}.
	
	\begin{conjecture}[Bounded-conductor finiteness conjecture in the function field case]\label{MTfinitenessconj}
		Let \(K\) be a global function field of characteristic \(p'\), with
		absolute Galois group \(G_K\).  Let \(n\) be a positive integer and let
		\(\mathfrak N\) be an effective \(\mathbf Q\)-divisor of \(K\).  Then there are only
		finitely many isomorphism classes of continuous semisimple representations
		\[
		\rho:G_K\longrightarrow\GL_n(\Fbar)
		\]
		whose Artin conductor divides \(\mathfrak N\) and which are \emph{geometric}; that
		is, the fixed field of \(\ker(\rho)\) contains no nontrivial constant field
		extension of \(K\).
	\end{conjecture}
	\begin{remark}
		\begin{enumerate}
			\item The geometric assumption cannot be removed. Indeed, consider the natural
			projection
			\(
			G_K\twoheadrightarrow
			\operatorname{Gal}(\overline{k_{0}}/k_{0})
			\simeq \widehat{\mathbf Z}
			\) where $k_{0}$ is the full constant subfield of $K$. For every integer $m$ prime to $p$, choose an element
			$\zeta_m\in \overline{\mathbf F}_p^\times$ of order $m$. Composing the above
			projection with the character $	\widehat{\mathbf Z}\longrightarrow
			\overline{\mathbf F}_p^\times,	1\longmapsto \zeta_m $, we obtain a continuous one-dimensional representation $	\chi_m:G_K\longrightarrow \overline{\mathbf F}_p^\times $ that is
			unramified at every place of $K$. Since the characters $\chi_{m}$ give infinitely many non-isomorphic one-dimensional representations with trivial Artin conductor, the statement is false for $n=1$ without the geometric condition.
			
			\item If $\overline{\mathbf F}_p$ is replaced by a finite field $\mathbf F_{p^r}$, then the finiteness statement follows immediately from the analogue of the Hermite--Minkowski
			theorem for global function fields. See \cite[Sect.~4, (a)--(b), p.~2534]{MoonTaguchi01}.
		\end{enumerate}
	\end{remark}
	
	Moon and Taguchi proved this conjecture under the additional assumption
	that the image of \(\rho\) is solvable
	\cite[Thm.~4(ii), p.~2533]{MoonTaguchi01}.  Furthermore, in the
	cross-characteristic case \(p\ne p'\), the first author proved the
	conjecture in arbitrary dimension when \(p\ne2\), and also in dimension
	\(n=2\) when \(p=2\); see \cite[Thm.~1.3, p.~2]{Luo26a}.  For related
	finiteness results for varieties over finite fields, see \cite[Thms.~1.3, 1.4 and~1.6, pp.~3--4]{Luo26b}.
	The present paper studies the equal-characteristic case \(p=p'\) of the
	conjecture when \(\mathfrak{N}=0\), namely, for representations with a
	trivial Artin conductor.
	
	\subsection{Main results}
	Let \(X\) be a smooth projective
	geometrically connected curve over a finite field $k$ of characteristic $p$, and write \(K\) for
	the function field of \(X\). The \emph{arithmetic and geometric fundamental
		groups} of \(X\) are \(\pi_1^{\et}(X)\) and
	\(\pi_1^{\et}(X_{\overline{k}})\), respectively. We choose geometric
	base points compatibly and omit them from the notation.
	There is an exact sequence of profinite groups
	\cite[Lem.~58.14.3]{Stacks}
	\begin{equation}\label{eq:fundamental-exact-sequence}
		1\longrightarrow \pi_1^{\et}(X_{\overline{k}})
		\longrightarrow \pi_1^{\et}(X)
		\longrightarrow \Gal(\overline{k}/k)
		\longrightarrow 1.
	\end{equation}
	\begingroup
	\emergencystretch=1em
	The canonical quotient
	\(
	G_K\twoheadrightarrow\pi_1^{\et}(X)
	\)
	identifies continuous representations of \(\pi_1^{\et}(X)\) with continuous
	representations of \(G_K\) that are unramified at every place of \(K\)
	\cite[Lem.~58.11.1 and Prop.~58.11.3]{Stacks}.
	Such representations have trivial Artin conductor, and any continuous
	representation \(\rho:\pi_1^{\et}(X)\longrightarrow\GL_n(\Fbar)\) has finite image.
	
	Such a representation is called \emph{geometric} in the sense of
	Conjecture~\ref{MTfinitenessconj} if it satisfies any of the following
	equivalent conditions:
	\begin{enumerate}
		\item[(i)]
		\[
		\rho\!\left(\pi_1^{\et}(X_{\overline{k}})\right)
		=
		\rho\!\left(\pi_1^{\et}(X)\right).
		\]
		\item[(ii)] The connected finite \'{e}tale cover of \(X\) corresponding
		to \(\ker(\rho)\) is geometrically connected over \(k\).
		\item[(iii)] The field \(k\) is algebraically closed in the function field
		of this cover; equivalently, that function field is regular over \(k\).
	\end{enumerate}
	
	In the rest of the paper, we will mainly use formulation (i) when
	referring to geometric representations.
	These equivalences follow from the Galois correspondence for finite
	\'{e}tale covers
	\cite[Thm.~58.6.2(1), (3) and Sect.~58.7]{Stacks}
	and the connectedness criteria for smooth proper curves over a perfect
	field \cite[Lems.~33.10.7 and~33.8.6]{Stacks}.
	
	Therefore, the following is precisely the trivial Artin conductor special case of the Moon--Taguchi conjecture \ref{MTfinitenessconj} in the
	equal characteristic case.
	\par\endgroup
	
	\begin{conjecture}\label{ass:naive}
		Let \(X\) be a smooth projective geometrically connected curve over a
		finite field \(k\) of characteristic \(p\). For every positive integer
		\(n\), there are only finitely many isomorphism classes of continuous
		semisimple geometric representations
		\[
		\rho:\pi_1^{\et}(X)\longrightarrow\GL_n(\Fbar).
		\]
	\end{conjecture}
	
	Our first main result establishes the conjectured finiteness statement under a characteristic-zero liftability hypothesis; see
	Definition~\ref{lifttocharzero} for the precise meaning of this condition.
	
	\begin{theorem}\label{conditionaltheorem}
		Let \(X\) be a smooth projective geometrically connected curve over a
		finite field \(k\) of characteristic \(p\). Then, for every positive
		integer \(n\), there are only finitely many isomorphism classes of
		continuous semisimple geometric representations
		\[
		\rho:\pi_1^{\et}(X)\longrightarrow\GL_n(\Fbar)
		\]
		which admit a lift to characteristic zero.
	\end{theorem}
	
	The liftability hypothesis holds whenever the image has order prime to \(p\)
	\cite[\S15.5, Prop.~43 and Exer.~15.9, pp.~128--129]{Serre77}. The proof of
	Theorem~\ref{conditionaltheorem} ultimately rests on Abe's Langlands
	correspondence for overconvergent \(F\)-isocrystals \cite[Thm.~4.2.2, pp.~1023--1024]{Abe18}.
	
	\begin{remark}
		For comparison, suppose that \(k\) has characteristic different from
		\(p\), while coefficients are still taken in \(\Fbar\), and that
		\(\rho\) is geometric and absolutely irreducible. By formulation (i),
		\(\rho\) and its restriction to \(\pi_1^{\et}(X_{\overline{k}})\)
		have the same image, so the restriction is absolutely irreducible.
		If \(p>2\) and \(p\nmid n\), then de Jong's lifting
		theorem, together with Gaitsgory's proof of de Jong's conjecture,
		implies that \(\rho\) admits a lift to characteristic zero; see
		\cite[Thm.~3.5 and Rem.~3.6(b), pp.~71--72]{deJong01} and
		\cite[Thm.~3.6, p.~166, and App.~A, pp.~181--184]{Gaitsgory07}.  By contrast, in the equal-characteristic setting
		\(\operatorname{char}(k)=p\) considered here, liftability is an additional condition.
	\end{remark}
	
	The liftability hypothesis cannot be removed: without it, the finiteness
	statement fails already in rank \(2\).
	Our second main result is the following theorem, which gives the first counterexample to the
	Moon--Taguchi bounded-conductor finiteness conjecture. 
	
	\begin{theorem}\label{thm:main}
		For every prime \(p\), there exists a smooth projective geometrically
		connected curve \(X\) over \(\Fp\) that admits, for every integer
		\(r\geq4\), a continuous absolutely irreducible geometric representation
		\[
		\pi_1^{\et}(X)\longrightarrow\GL_2(\Fbar)
		\]
		with image isomorphic to \(\SL_2(\mathbf F_{p^r})\). Moreover, there
		are infinitely many pairwise geometrically non-isomorphic choices for
		\(X\). In particular, Conjecture~\ref{ass:naive} and hence
		Conjecture~\ref{MTfinitenessconj} fail for \(n=2\).
	\end{theorem}
	
	The proof unfolds in the following four steps.
	
	First, we construct a smooth projective geometrically connected curve
	over \(\Fp\) as a fiber of a moduli space of \(D\)-elliptic sheaves,
	for a suitable quaternion division algebra \(D\) over \(\Fp(T)\).
	Adding an auxiliary prime level of degree \(r\) gives a finite
	\'{e}tale \(\GL_2(\mathbf F_{p^r})\)-torsor over this fixed curve,
	and hence a two-dimensional representation of its \'{e}tale fundamental
	group.
	
	Second, to determine the geometric image, we use the determinant
	morphism to the rank-one moduli space. Its geometrically irreducible
	fibers, together with its equivariance under the level-group actions,
	identify the geometric monodromy group with \(\SL_2(\mathbf F_{p^r})\);
	see Proposition~\ref{prop:modular-input} and
	Theorem~\ref{thm:level-components}.
	
	Third, we use a rank-one Frobenius calculation to choose, for every
	\(r\ge 4\), an auxiliary level whose Frobenius determinant is a square
	in \(\mathbf F_{p^r}\). It is worth noting that the common
	characteristic \(p\) of the curve and the representation coefficients
	lets us use this level residue field for the scalar twist as well;
	see Remark~\ref{rem:equal-characteristic} for further details.
	A suitable scalar twist by a character of the absolute Galois group of
	\(\Fp\) then makes the determinant character trivial without changing
	the geometric image. The arithmetic and geometric images are therefore
	both \(\SL_2(\mathbf F_{p^r})\), and the resulting representations are
	absolutely irreducible.
	
	Finally, Proposition~\ref{prop:monodromy-towers} supplies a tower of
	finite \'{e}tale covers that simultaneously preserves all these
	monodromy groups. The genera increase along the tower, giving infinitely
	many pairwise geometrically nonisomorphic curves as required in
	Theorem~\ref{thm:main}. The details of the construction and proof are
	given in Sections~3--6.
	
	Frey, Kani and V\"olklein \cite[Sect.~4.2, Cor.~4.10, pp.~93--95]{FKV99} constructed
	examples of global function fields admitting infinite
	everywhere unramified regular Galois extensions whose Galois
	groups are products of groups \(\mathrm{PSL}_d(\mathbf F_{p_{i}})\),
	with \(p_{i}\) ranging over an infinite set of primes. In \cite[the final paragraph of Sect.~3, p.~2533]{MoonTaguchi01}, Moon and Taguchi asked whether such an extension could instead be
	obtained using groups \(\mathrm{PSL}_d(\mathbf F_q)\), with
	\(q\) ranging over distinct powers of a fixed prime. Theorem~\ref{thm:main} answers this existence question for \(d=2\)
	in equal characteristic. Together with the cross-characteristic finiteness theorem
	\cite[Thm.~1.3, p.~2]{Luo26a}, this gives the following dichotomy.

	\begin{corollary}\label{cor:infinite-product}
		Let \(p\) be a prime number.
		\begin{enumerate}
			\item There exist a global function field \(K\) with full constant
			field \(\Fp\) and an infinite Galois extension \(L/K\), unramified
			at every place and regular over \(\Fp\), such that
			\[
			\Gal(L/K)\cong
			\prod_{r\geq4}\PSL_2(\mathbf F_{p^r})
			\]
			as profinite groups.
			
			\item Let \(p'\ne p\) be a prime number, let \(K\) be a global
			function field of characteristic \(p'\), let \(k_0\) be its full
			constant field, and let \(S\) be an infinite set of powers of \(p\).
			Then there is no infinite Galois extension \(L/K\), unramified at every
			place and regular over \(k_0\), such that
			\[
			\Gal(L/K)\cong
			\prod_{q\in S}\PSL_2(\mathbf F_q)
			\]
			as profinite groups.
		\end{enumerate}
		Here regularity over the full constant field means that this field is
		algebraically closed in \(L\).
	\end{corollary}
	
	\begin{remark}
		The construction is not intrinsically restricted to rank two.
		We expect higher-rank analogues of Theorem~\ref{thm:main} and
		Corollary~\ref{cor:infinite-product}(1) for every fixed \mbox{\(n\geq 2\)},
		still over smooth projective geometrically connected curves.
		In higher rank, however, the relevant moduli spaces are
		higher-dimensional. One must establish the required level monodromy
		and then pass to a curve that preserves all these monodromy groups
		simultaneously, independently of the residue degree.
		
		There is also an arithmetic obstacle to extending the argument
		directly. Removing the determinant by a scalar twist over the
		original coefficient field requires an \(n\)-th-power condition,
		but such powers may all lie in a proper subfield, preventing the
		choice of auxiliary primes used here. Overcoming this obstruction
		while keeping one curve for all the required residue degrees calls
		for further arguments. These additional geometric and arithmetic
		steps are the reason we restrict the present paper to rank two.
		We expect them to be accessible within the same general approach
		and intend to develop the higher-rank construction in future work.
		
		For Corollary~\ref{cor:infinite-product}(2), the cross-characteristic
		finiteness argument extends to every fixed \(n\geq 2\) when \(p\)
		is odd. We also expect the higher-rank assertion in coefficient
		characteristic two, although this case requires a separate argument
		beyond the finiteness theorem used here.
	\end{remark}
	
	The two main theorems also imply that, on each curve supplied
	by Theorem~\ref{thm:main}, the resulting representations fail to admit
	characteristic-zero lifts for all sufficiently large residue degrees.
	
	\begin{corollary}\label{cor:nonliftable-family}
		Let \(p\) be a prime number. Let \(X\) and
		\((\rho_r)_{r\geq 4}\) be a curve over \(\Fp\) and the corresponding
		family of representations constructed in Theorem~\ref{thm:main}.
		Then there exists an integer \(r_0=r_0(X)\geq 4\) such that,
		for every \(r\geq r_0\), the representation \(\rho_r\)
		admits no lift to characteristic zero.
	\end{corollary}

	\subsection{Organization and notation}
	The paper is organized as follows.  In Section~2, we prove
	Theorem~\ref{conditionaltheorem}.  Sections~3--6 prove
	Theorem~\ref{thm:main}, and
	Section~\ref{sec:infinite-product} proves
	Corollary~\ref{cor:infinite-product} and Corollary~\ref{cor:nonliftable-family}.
	
	Throughout this paper, \(p\) is a prime number, \(\Fp\) is the finite field of order \(p\) and \(\mathbf{Q}_{p}\) is the field of $p$-adic numbers. For any positive integer $m$, we denote by $\mathbf F_{p^m}$ the finite field of order $p^{m}$. If $F$ is a field, then we denote by $\overline{F}$ an algebraic closure of $F$. We use the \emph{arithmetic Frobenius}
	\[
	\Frob\in\Gal(\Fbar/\Fp),\qquad a\longmapsto a^p;
	\]
	the geometric Frobenius is \(\Frob^{-1}\).  Fundamental groups act on
	geometric fibers on the left, while torsor actions are written on the
	right.
	
	\section{Proof of Theorem \ref{conditionaltheorem}}
	Let \(X\) be a smooth projective geometrically connected curve over a finite field of characteristic $p$. We first recall the notion of characteristic-zero lifts.
	\begin{definition}\label{lifttocharzero}
		We say that a continuous representation
		\[
		\rho:\pi_1^{\et}(X)\longrightarrow\GL_n(\Fbar)
		\]
		admits \emph{a lift to characteristic zero} if there exist a finite extension
		$E/\mathbf Q_p$, with ring of integers $\mathcal O_E$ and residue field
		$\kappa_E$, an embedding $\kappa_E\hookrightarrow \Fbar$, and a continuous
		representation
		\[
		\widetilde\rho: \pi_1^{\et}(X)\longrightarrow \GL_n(\mathcal O_E)
		\]
		such that, after extension of scalars through $\kappa_E\hookrightarrow \Fbar$,
		the reduction of $\widetilde\rho$ modulo the maximal ideal of $\mathcal O_E$ is isomorphic to $\rho$. Here the group $\GL_{n}(\mathcal{O}_{E})$ is equipped with the $p$-adic topology induced from the topological ring $\mathcal{O}_{E}$.
	\end{definition}
	
	Our argument uses the following consequence of the Langlands correspondence.
	\begin{theorem}\label{Abetheorem}
		Let \(X\) be a smooth projective geometrically connected curve over a finite field $k$ of characteristic $p$ and let $n$ be a positive integer. Then there are only finitely many isomorphism classes of irreducible continuous representations
		\[
		\rho:\pi_1^{\et}(X)\to \mathrm{GL}_n(\overline{\mathbf Q}_{p})
		\]
		up to twist by a character of $\pi_1^{\et}(\operatorname{Spec}k)$.
	\end{theorem}
	\begin{proof}
		Since \(X\) is proper, the Katz--Crew correspondence identifies
		continuous \(p\)-adic representations of \(\pi_1^{\et}(X)\) with
		unit-root overconvergent \(F\)-isocrystals on \(X\) with coefficients
		in \(\overline{\mathbf Q}_p\); see
		\cite[Prop.~2.13, p.~8]{Kedlaya22}. If two unit-root objects differ by a constant rank-one twist, then that twist has slope zero and is therefore unit-root. Under this correspondence, such twists correspond to continuous
		characters of \(\pi_1^{\et}(\operatorname{Spec}k)\). The assertion is
		\cite[Cor.~3.6, p.~13]{Kedlaya22}.  Its proof combines the Langlands
		correspondence for coefficient objects \cite[Thm.~3.4, p.~13]{Kedlaya22}, whose
		crystalline case is \cite[Thm.~4.2.2, pp.~1023--1024]{Abe18}, with Harder's
		support theorem and finite-dimensionality corollary
		\cite[Thm.~1.2.1 and Cor.~1.2.3, pp.~255--256]{Har74}.
	\end{proof}
	
	We now prove Theorem~\ref{conditionaltheorem}.
	\begin{proof}[Proof of Theorem \ref{conditionaltheorem}]
		Let
		\(
		\rho:\pi_1^{\et}(X)\longrightarrow\GL_n(\Fbar)
		\)
		be a continuous semisimple geometric representation admitting a lift
		to characteristic zero. Choose such a lift
		\(
		\widetilde{\rho}:\pi_1^{\et}(X)\longrightarrow\GL_n(\mathcal O_E)
		\). Let \(V\) denote the semisimplification of
		\(
		\widetilde{\rho}\otimes_{\mathcal O_E}E.
		\)
		Then \(V\) is a semisimple \(E\)-representation of \(\pi_1^{\et}(X)\).
		Reducing any \(\pi_1^{\et}(X)\)-stable lattice in \(V\) modulo the
		maximal ideal of \(\mathcal O_E\) and then semisimplifying gives
		a representation with the same characteristic polynomials as \(\rho\).
		Both reductions factor through a common finite quotient of
		\(\pi_1^{\et}(X)\), so the Brauer--Nesbitt theorem identifies them
		\cite[\S18.2, Cor.~1 and Exer.~18.1, pp.~149--150]{Serre77}.
		After enlarging \(E\) if necessary, write
		\[
		V=\bigoplus_i V_i
		\]
		as a direct sum of absolutely irreducible \(E\)-representations. By Theorem \ref{Abetheorem}, there exists a finite set \(\mathcal S\) of irreducible
		\(p\)-adic representations of \(\pi_1^{\et}(X)\), of dimensions at most \(n\),
		such that every \(V_i\) is of the form
		\[
		W\otimes\chi,
		\]
		where \(W\in\mathcal S\) and
		\(
		\chi:\pi_1^{\et}(\operatorname{Spec}k)
		\longrightarrow\overline{\mathbf Q}_p^\times
		\)
		is a continuous character.
		
		For each \(W\in\mathcal S\), choose a stable lattice and let
		\(\overline W^{\,\mathrm{ss}}\) denote its semisimplified
		residual representation.
		By continuity, each \(\chi\) takes values in units, so reducing the
		corresponding tensor-product lattices and semisimplifying expresses
		\(\rho\) as a direct sum of constant twists of the fixed
		\(\overline W^{\,\mathrm{ss}}\).
		Since \(\mathcal S\) is finite, only finitely many irreducible mod \(p\)
		representations occur as irreducible constituents of the various
		\(\overline W^{\,\mathrm{ss}}\). Consequently, there are only finitely many
		possibilities for the restriction
		\(
		\rho|_{\pi_1^{\et}(X_{\overline{k}})},
		\)
		because twists by characters of
		\(\pi_1^{\et}(\operatorname{Spec}k)\) are trivial on \(\pi_1^{\et}(X_{\overline{k}})\).
		It follows that the finite groups
		\(
		\rho(\pi_1^{\et}(X_{\overline{k}}))
		\)
		belong, up to conjugacy, to a finite collection.
		
		Now fix the reductions of the \(W\)'s occurring in \(\rho\). We may write
		\[
		\rho\simeq
		\bigoplus_i\left(R_i\otimes\overline{\chi}_i\right),
		\]
		where the \(R_i\)'s belong to a fixed finite collection and each
		\(\overline{\chi}_i\) factors through
		\(
		\pi_1^{\et}(\operatorname{Spec}k)\simeq\widehat{\mathbf Z}.
		\)
		Let \(F\in\pi_1^{\et}(X)\) be an arithmetic Frobenius and put
		\[
		A_i=R_i(F),
		\qquad
		c_i=\overline{\chi}_i(F).
		\]
		Then
		\[
		\rho(F)=\operatorname{diag}(c_1A_1,\ldots,c_rA_r).
		\]
		Since
		\[
		\rho(\pi_1^{\et}(X_{\overline{k}}))=\rho(\pi_1^{\et}(X)),
		\]
		we have
		\(
		\rho(F)\in\rho(\pi_1^{\et}(X_{\overline{k}}))
		\). For fixed \(R_1,\ldots,R_r\), the group on the right is fixed and finite, and
		the matrix \(\rho(F)\) uniquely determines \((c_1,\ldots,c_r)\). Hence only
		finitely many such tuples can occur. Finally, each character
		\(
		\overline{\chi}_i:
		\pi_1^{\et}(\operatorname{Spec}k)
		\longrightarrow\Fbar^\times
		\)
		is uniquely determined by its value on Frobenius. Thus only finitely many
		tuples \((\overline{\chi}_1,\ldots,\overline{\chi}_r)\) occur, and hence only
		finitely many isomorphism classes of \(\rho\) occur.
	\end{proof}
	
	\section{Components, monodromy, and scalar twists}
	
	The representations in Theorem~\ref{thm:main} will come from finite
	\'{e}tale level covers.  To turn such a cover into a geometric
	representation, we first read its geometric
	monodromy from its connected components and then remove the determinant
	contributed by arithmetic Frobenius.  The last result of the section shows
	how to preserve all these monodromy groups while replacing one base curve
	by infinitely many finite \'{e}tale covers.
	
	To determine the geometric image in Theorem~\ref{thm:level-components},
	we will compute the stabilizer of a connected component of the covering
	curve after base change to \(\Fbar\). The following lemma justifies
	this approach.
	
	\begin{lemma}\label{lem:components}
		Let \(X\) be a smooth geometrically connected curve over \(\Fp\), let \(G\) be a finite constant
		group, and let \(Y\to X\) be a finite \'{e}tale right \(G\)-torsor.  Choose a
		geometric point \(\bar x\to X\), a point \(\bar y\in Y_{\bar x}\), and let
		\[
		\theta:\pi_1^{\et}(X)\longrightarrow G
		\]
		be the monodromy homomorphism defined by
		\(\gamma(\bar y)=\bar y\,\theta(\gamma)\).
		Put \(H=\theta(\pi_1^{\et}(X_{\Fbar}))\).  Then the geometric components of \(Y\) form
		the right \(G\)-set \(H\backslash G\).  The stabilizer in \(G\) of the
		component containing \(\bar y\) is exactly \(H\).
	\end{lemma}
	
	\begin{proof}
		Identify \(Y_{\bar x}\) with the regular right \(G\)-set \(G\) by
		\(g\mapsto\bar y g\).  The fundamental group action commutes with the
		right \(G\)-action and is therefore left multiplication through \(\theta\).
		The \(\pi_1^{\et}(X_{\Fbar})\)-orbits in \(Y_{\bar x}\) are consequently the
		subsets \(Hg\). Applying the finite \'etale covering correspondence
		to \(Y_{\Fbar}\to X_{\Fbar}\), whose base is connected, identifies
		these orbits with the connected components of \(Y_{\Fbar}\)
		\cite[Thm.~58.6.2(1)]{Stacks}.
		Thus \(\pi_0(Y_{\Fbar})=H\backslash G\) as right \(G\)-sets.
		The component containing \(\bar y\) corresponds to the coset \(H\),
		and \(Hg=H\) holds exactly when \(g\in H\).
	\end{proof}
	
	In the modular torsors considered below, the component set is itself
	a torsor under a quotient of the acting group \(G\). The preceding
	lemma then identifies the geometric monodromy with the kernel of
	the corresponding quotient homomorphism.
	
	\begin{corollary}\label{cor:stabilizer}
		In the situation of Lemma~\ref{lem:components}, suppose that
		\(\pi_0(Y_{\Fbar})\) is a torsor under a finite group \(Q\), and that the
		right \(G\)-action on components is translation through a surjection
		\(\nu:G\to Q\).  Then
		\[
		\theta\bigl(\pi_1^{\et}(X_{\Fbar})\bigr)=\ker(\nu).
		\]
	\end{corollary}
	
	\begin{proof}
		Translation makes the stabilizer of every point of the \(Q\)-torsor equal
		to \(\ker(\nu)\).  Lemma~\ref{lem:components} identifies the same stabilizer
		with the geometric monodromy image.
	\end{proof}
	
	This criterion will identify the geometric image with \(\SL_2(k)\).  The
	arithmetic image may still be larger because arithmetic Frobenius need not
	have determinant one.  The next lemma shows that a square Frobenius
	determinant can be removed by a character of the constant field.
	
	\begin{lemma}\label{lem:twist}
		Let \(X\) be a smooth geometrically connected curve over \(\Fp\), let
		\(k\) be a finite field of characteristic $p$, and let
		\(
		\theta:\pi_1^{\et}(X)\longrightarrow\GL_2(k)
		\)
		be a continuous representation such that 
		\(
		\theta\bigl(\pi_1^{\et}(X_{\Fbar})\bigr)=\SL_2(k)
		\).
		Suppose that
		\[
		\det\theta(\Frob)\in(k^\times)^2,
		\]
		where a lift of \(\Frob\) to \(\pi_1^{\et}(X)\) is understood. Then there is a
		continuous character
		\(
		\chi:\pi_1^{\et}(X)\longrightarrow k^\times
		\)
		which is trivial on \(\pi_1^{\et}(X_{\Fbar})\) and for which
		\[
		\rho(\gamma):=\chi(\gamma)\theta(\gamma)
		\]
		satisfies
		\[
		\rho\bigl(\pi_1^{\et}(X_{\Fbar})\bigr)
		=
		\rho\bigl(\pi_1^{\et}(X)\bigr)
		=
		\SL_2(k).
		\]
	\end{lemma}
	
	\begin{proof}
		Set \(\delta=\det\theta\).  Since \(\delta\) is trivial on
		\(\pi_1^{\et}(X_{\Fbar})\), the exact sequence
		\eqref{eq:fundamental-exact-sequence} makes it a continuous character of
		\(\Gal(\Fbar/\Fp)\cong\widehat{\mathbf Z}\).  Put
		\(a=\delta(\Frob)\), and choose \(b\in k^\times\) with \(b^2=a^{-1}\).
		Because \(b\) has finite order, the homomorphism
		\[
		\widehat{\mathbf Z}\longrightarrow k^\times,
		\qquad 1\longmapsto b,
		\]
		is continuous.  Pull it back to \(\pi_1^{\et}(X)\) and call the resulting
		character \(\chi\).  The character \(\chi^2\delta\) is trivial on the
		topological generator \(\Frob\), hence is trivial.  Therefore
		\(\det\rho=1\).
		
		Since the character \(\chi\) is trivial on \(\pi_1^{\et}(X_{\Fbar})\), the geometric image of \(\rho\) is \(\SL_2(k)\). The full image is contained in
		\(\SL_2(k)\) because its determinant is one, and it contains the geometric image.  Both images are therefore equal to \(\SL_2(k)\).
	\end{proof}
	
	After the twist, the resulting representation is the natural
	two-dimensional representation of \(\SL_2(k)\).  The main results require it
	to be absolutely irreducible, which follows from the following elementary
	fact.
	
	\begin{lemma}\label{lem:natural-irred}
		For every finite field \(k\) of characteristic $p$, the natural representation of
		\(\SL_2(k)\) on the two-dimensional vector space \(k^2\) is absolutely irreducible.
	\end{lemma}
	
	\begin{proof}
		Extend scalars to \(\Fbar\).  The two elements
		\[
		u_+=
		\begin{pmatrix}1&1\\0&1\end{pmatrix},
		\qquad
		u_-=
		\begin{pmatrix}1&0\\1&1\end{pmatrix}
		\]
		belong to \(\SL_2(k)\).  If a line is invariant under \(u_+\), then the
		restriction of \(u_+\) to that line has eigenvalue \(1\), so the line is
		contained in \(\ker(u_+-1)=\Fbar e_1\).  Similarly, a line invariant under
		\(u_-\) must be \(\Fbar e_2\).  No line is invariant under both elements.
		Thus there is no \(\SL_2(k)\)-invariant line in \(\Fbar^2\), which proves
		absolute irreducibility.
	\end{proof}
	
	The preceding lemmas will be applied below to construct all the required
	representations on a single curve.  The next proposition shows how to pass
	from that curve to infinitely many others without changing any of the
	monodromy groups.
	
	\begin{proposition}
		\label{prop:monodromy-towers}
		Let \(X\) be a smooth projective geometrically connected curve of genus
		\(g\) over \(\mathbf F_q\), and let \(\lambda\) be a rational prime
		different from the characteristic of \(\mathbf F_q\).  Then there exist
		smooth projective geometrically connected curves \(X_j\) over
		\(\mathbf F_q\) and finite \'{e}tale covers
		\[
		X_j\longrightarrow X,\qquad j\geq1,
		\]
		depending only on \(X\) and \(\lambda\), with the following properties.
		\begin{enumerate}
			\item For every \(j\geq1\),
			\[
			\deg(X_j/X)=\lambda^{2gj},
			\qquad
			g(X_j)=1+\lambda^{2gj}(g-1).
			\]
			
			\item If the genus of \(X\) is at least \(2\), then the curves
			\(X_j\) are pairwise geometrically nonisomorphic.
			
			\item Let \(G\) be a finite perfect group and let
			\[
			\rho:\pi_1^{\et}(X)\longrightarrow G
			\]
			be a continuous homomorphism whose restriction to \(\pi_1^{\et}(X_{\overline{\mathbf F}_q})\)
			is surjective.  Then, for every \(j\geq1\),
			\[
			\rho\bigl(\pi_1^{\et}((X_j)_{\overline{\mathbf F}_q})\bigr)
			=
			\rho\bigl(\pi_1^{\et}(X_j)\bigr)
			=G.
			\]
		\end{enumerate}
	\end{proposition}
	
	\begin{proof}
		Put \(\Gamma=\pi_1^{\et}(X)\) and \(\Pi=\pi_1^{\et}(X_{\overline{\mathbf F}_q})\).  The construction is
		guided by the last assertion.  To preserve a perfect quotient of \(\Pi\),
		it is enough to pass to subgroups that still contain the closed commutator
		subgroup \(\overline{[\Pi,\Pi]}\).  We therefore obtain the covers by
		shrinking only the abelianization of \(\Pi\).
		
		Fix \(j\geq1\) and put \(n=\lambda^j\).  After choosing an identification
		\(\mu_n\simeq\mathbf Z/n\mathbf Z\), the Kummer sequence on
		\(X_{\overline{\mathbf F}_q}\) gives
		\cite[Lem.~59.69.1]{Stacks}
		\begin{align*}
			\operatorname{Hom}_{\mathrm{cont}}(\Pi,\mathbf Z/n\mathbf Z)
			&=H^1_{\et}(X_{\overline{\mathbf F}_q},\mathbf Z/n\mathbf Z)\\
			&\simeq H^1_{\et}(X_{\overline{\mathbf F}_q},\mu_n)\\
			&\simeq \operatorname{Pic}(X_{\overline{\mathbf F}_q})[n]\\
			&=\operatorname{Jac}(X)[n](\overline{\mathbf F}_q)\\
			&\simeq(\mathbf Z/n\mathbf Z)^{2g}.
		\end{align*}
		The Kummer sequence has no contribution from global units because
		\(\overline{\mathbf F}_q^\times\) is \(n\)-divisible.  Since \(\lambda\)
		is different from the characteristic of \(\mathbf F_q\), the group of
		\(n\)-torsion points of the Jacobian is isomorphic to
		\((\mathbf Z/n\mathbf Z)^{2g}\)
		\cite[Prop.~39.9.11]{Stacks}.  Pontryagin duality therefore gives
		\begin{equation}\label{eq:abelian-level-index}
			\Pi^{\mathrm{ab}}/\lambda^j\Pi^{\mathrm{ab}}
			\simeq(\mathbf Z/\lambda^j\mathbf Z)^{2g}.
		\end{equation}
		Define the open characteristic subgroup
		\[
		K_j=\ker\!\left(
		\Pi\longrightarrow
		\Pi^{\mathrm{ab}}/\lambda^j\Pi^{\mathrm{ab}}
		\right).
		\]
		
		We next lift \(K_j\) to \(\Gamma\) in a way that does not enlarge the
		constant field.  The constant-field exact sequence admits a continuous
		section \(s:\widehat{\mathbf Z}\to\Gamma\): indeed,
		\(\widehat{\mathbf Z}\) is the free profinite group on one generator, so a
		lift of that generator to \(\Gamma\) determines a section
		\cite[\S3.3, Ex.~3.3.8(a), p.~90]{RibesZalesskii10}.  Fix such a
		section once and for all.  Since \(K_j\) is characteristic in the normal
		subgroup \(\Pi\), the subgroup \(s(\widehat{\mathbf Z})\) normalizes
		\(K_j\), and hence
		\[
		\Gamma_j=K_j s(\widehat{\mathbf Z})
		\]
		is a subgroup of \(\Gamma\).  The splitting writes
		\(\Gamma=\Pi\rtimes s(\widehat{\mathbf Z})\), so
		\[
		\Gamma_j\cap\Pi=K_j,
		\qquad
		[\Gamma:\Gamma_j]=[\Pi:K_j]=\lambda^{2gj}
		\]
		by \eqref{eq:abelian-level-index}.  Thus \(\Gamma_j\) is open and, after
		choosing compatible geometric base points, corresponds to a connected finite
		\'{e}tale cover \(X_j\to X\) of degree \(\lambda^{2gj}\)
		\cite[Thm.~58.6.2]{Stacks}.  Moreover,
		\(\Gamma_j\to\widehat{\mathbf Z}\) is surjective, equivalently
		\(\Gamma=\Pi\Gamma_j\); hence \(X_j\) is geometrically connected and
		\(\pi_1^{\et}((X_j)_{\overline{\mathbf F}_q})\) identifies with \(K_j\), up to conjugacy.  Being finite
		\'{e}tale over \(X\), the curve \(X_j\) is smooth and projective.
		
		The unramified Riemann--Hurwitz formula
		\cite[Lem.~53.12.2]{Stacks} now gives
		\[
		2g(X_j)-2=\lambda^{2gj}(2g-2),
		\]
		which is the genus formula in part~(1).  If the genus of \(X\) is at least
		\(2\), these genera are strictly increasing with \(j\); this proves
		part~(2).
		
		It remains to prove part~(3).  Let \(G\) and \(\rho\) be as in that part.
		Since \(K_j\) contains \(\overline{[\Pi,\Pi]}\), the surjectivity of
		\(\rho|_\Pi\) and the perfectness of \(G\) give
		\[
		\rho(K_j)
		\supseteq
		\rho\bigl(\overline{[\Pi,\Pi]}\bigr)
		=[\rho(\Pi),\rho(\Pi)]
		=[G,G]
		=G.
		\]
		Thus \(\rho(K_j)=G\).  The arithmetic image of \(\Gamma_j\) contains
		this geometric image and is itself contained in \(G\), so both images are
		equal to \(G\).  The construction of \(K_j\) and \(\Gamma_j\) did not use
		\(G\) or \(\rho\); consequently the same tower has this property for every
		finite perfect quotient under consideration.
	\end{proof}
	
	\section{Geometry of the moduli spaces and a rank-one calculation}
	
	Section~3 reduces the construction to two geometric questions: what are the
	components of the level cover, and how does arithmetic Frobenius act on
	them?  Both questions are answered by comparing the quaternionic moduli curve
	with its rank-one counterpart through the determinant morphism.
	
	Put
	\[
	A=\Fp[T],\qquad C=\mathbf P^1_{\Fp},
	\]
	and let \(\infty\) be the point corresponding to \(1/T\).  Write
	\(\mathcal O_C\) for the structure sheaf of \(C\).  Let \(D\) be a quaternion
	algebra over \(\Fp(T)\), split at \(\infty\) and ramified at some finite
	place, and let
	\(\mathcal D\) be a maximal \(\mathcal O_C\)-order in \(D\).  Let
	\(\Ram(D)\subset |C|\) denote the finite set of closed points \(v\) such that
	\[
	D\otimes_{\Fp(T)}\Fp(T)_v
	\not\simeq \mathrm{M}_2\bigl(\Fp(T)_v\bigr).
	\]
	The points in \(\Ram(D)\) are called the \emph{ramification points} of \(D\);
	equivalently, at each such point \(v\), the completion of \(D\) is a quaternion
	division algebra rather than a \(2\times2\) matrix algebra.
	
	For a nonzero ideal \(\mathfrak i\subset A\), let
	\(I=\operatorname{Spec}(A/\mathfrak i)\subset C\setminus\{\infty\}\).  We
	shall use the same letter \(I\) for the ideal and the corresponding finite closed
	subscheme; thus expressions such as \(A/I\) and \(I\ell\) refer to ideals.
	An inclusion \(I\subset I'\) between levels means inclusion of closed
	subschemes, equivalently \(\mathfrak i'\subset\mathfrak i\) for their defining
	ideals.  For a nonempty finite level
	\[
	I\subset C\setminus\bigl(\Ram(D)\cup\{\infty\}\bigr),
	\]
	write
	\[
	C'_I=C\setminus\bigl(I\cup\Ram(D)\cup\{\infty\}\bigr),
	\qquad
	\mathcal D_I=H^0(I,\mathcal D|_I).
	\]
	
	To make the determinant morphism and the level actions precise, we first
	recall the two moduli spaces involved.
	
	\begin{definition}[The moduli schemes \(M_I^D\) and
		\(M_I^{\mathcal O_C}\)]\label{def:moduli-scheme}
		Let \(\mathcal B\) be \(\mathcal D\) or \(\mathcal O_C\), put
		\(d_{\mathcal B}=2\) or \(1\), respectively, and set
		\[
		U_{\mathcal B,I}=
		\begin{cases}
			C'_I,&\mathcal B=\mathcal D,\\
			C\setminus(I\cup\{\infty\}),&\mathcal B=\mathcal O_C.
		\end{cases}
		\]
		For \(S\to U_{\mathcal B,I}\), put
		\(\tau=(\mathrm{id}_C\times\operatorname{Frob}_S)^*\).  A
		\emph{full level-\(I\) \(\mathcal B\)-elliptic sheaf} is a tuple
		\[
		\bigl((\mathcal E_i,j_i,t_i)_{i\in\mathbf Z},\iota\bigr),
		\]
		where \(\mathcal E_i\) is locally free of rank \(d_{\mathcal B}^2\) on
		\(C\times S\), with a right \(\mathcal B\)-action, and
		\[
		j_i:\mathcal E_i\hookrightarrow\mathcal E_{i+1},
		\qquad
		t_i:\tau\mathcal E_i\hookrightarrow\mathcal E_{i+1}
		\]
		are injective \(\mathcal B\)-linear maps.  They satisfy
		\[
		j_i t_{i-1}=t_i\tau(j_{i-1}),
		\qquad
		\mathcal E_{i+d_{\mathcal B}}=\mathcal E_i(\infty),
		\]
		and \(j_{i+d_{\mathcal B}-1}\cdots j_i\) is the natural inclusion into
		\(\mathcal E_i(\infty)\).  The cokernels of \(j_i\) and \(t_i\) are the direct images of
		locally free \(\mathcal O_S\)-modules of rank \(d_{\mathcal B}\)
		along the section \(s\mapsto(\infty,s)\) and the graph of the given
		morphism \(S\to C\), respectively
		\cite[Def.~(2.2) and Rems.~(2.3)(b)--(c), pp.~223--224]{LRS93}.
		Since the characteristic and \(\infty\) are disjoint from \(I\), the
		restrictions of every \(j_i\) and \(t_i\) to \(I\times S\) are
		isomorphisms.  The maps \(j_i|_{I\times S}\) identify all
		\(\mathcal E_i|_{I\times S}\) with a single locally free module
		\(\mathcal E_I\); the commutative squares above imply that the restricted
		maps \(t_i\) induce one map
		\[
		t_I:\tau\mathcal E_I\xrightarrow{\ \sim\ }\mathcal E_I.
		\]
		A \emph{full level structure} is a \(\mathcal B|_I\)-linear isomorphism
		\[
		\iota:\mathcal B|_I\boxtimes\mathcal O_S
		\xrightarrow{\ \sim\ }\mathcal E_I
		\]
		satisfying \(t_I\circ\tau(\iota)=\iota\), where
		\(\tau(\mathcal B|_I\boxtimes\mathcal O_S)\) is canonically identified
		with \(\mathcal B|_I\boxtimes\mathcal O_S\).  Thus the target of
		\(\iota\) is independent of the index; this is the convention of
		\cite[(2.5) and Def. (2.7), pp.~225--226]{LRS93}.
		For \(g\in H^0(I,\mathcal B|_I)^\times\), the right action on level
		structures is \(\iota\cdot g=\iota\circ L_g\), where \(L_g(x)=gx\).
		
		Let \(\mathcal E\!ll_I^{\mathcal B}\) denote the resulting fppf moduli
		stack.  The index shift \(i\mapsto i+1\) acts by reindexing the three
		families above and transporting \(\mathcal E_I\) and its trivialization.
		The quotient fppf sheaf \(\mathcal E\!ll_I^{\mathcal B}/\mathbf Z\) is
		represented by \(M_I^{\mathcal B}\) over
		\(U_{\mathcal B,I}\).  We write \(M_I^{\mathcal D}=M_I^D\).  For
		\(\mathcal B=\mathcal O_C\), this is equivalently the moduli scheme of
		rank-one Drinfeld \(A\)-modules with full level \(I\).  See
		\cite[\S2, pp.~223--227, and Thm. (5.1), p.~241]{LRS93} and
		\cite[\S3, pp.~121--122]{Pap09}.
	\end{definition}
	
	Since \(C'_I\subset C\setminus(I\cup\{\infty\})\), define the restriction
	to the common characteristic base by
	\[
	M_{I,C'_I}^{\mathcal O_C}
	:=M_I^{\mathcal O_C}
	\times_{C\setminus(I\cup\{\infty\})}C'_I.
	\]
	For \(\mathcal B\in\{\mathcal D,\mathcal O_C\}\) and a closed point
	\(v\in U_{\mathcal B,I}\), set
	\begin{equation}\label{eq:modular-fiber-notation}
		M_{I,v}^{\mathcal B}
		:=M_I^{\mathcal B}\times_{U_{\mathcal B,I}}\operatorname{Spec}\kappa(v),
	\end{equation}
	and write \(M_{I,v}^D\) when \(\mathcal B=\mathcal D\).
	
	The reduced norm induces a canonical homomorphism
	\[
	\operatorname{Nrd}_I:\mathcal D_I^\times\longrightarrow(A/I)^\times.
	\]
	At every finite point outside \(\Ram(D)\), the completed local order
	is a matrix algebra over the completed local ring
	\cite[(1.4)--(1.5), pp.~222--223]{LRS93}. Reducing these local splittings
	modulo the prime-power factors of \(I\) and applying the Chinese
	remainder theorem shows that, whenever \(I\) is disjoint from
	\(\Ram(D)\), the algebra \(\mathcal D_I\) is split Azumaya.
	After choosing a splitting
	\(\mathcal D_I\simeq \mathrm{M}_2(A/I)\), the map \(\operatorname{Nrd}_I\) becomes
	the ordinary matrix determinant
	\cite[(9.6)--(9.8), p.~116]{Rei03}. We use \(\det\) only after such a
	splitting has been fixed.
	
	The proof of Theorem~\ref{thm:main} uses three features of these moduli
	spaces: \(M_I^D\to C'_I\) is a smooth projective family of curves,
	changing the level gives a finite \'{e}tale torsor, and the determinant
	morphism has irreducible fibers. The following proposition provides
	all three.
	
	\begin{proposition}\label{prop:modular-input}
		The moduli schemes introduced above have the following properties.
		\begin{enumerate}
			\item The morphism \(M_I^D\to C'_I\) is smooth and projective of pure
			relative dimension \(1\).
			\item If \(I\subset I'\) are nonempty finite levels in the closed-subscheme
			sense fixed above, then the level-forgetting
			map
			\begin{equation}\label{eq:correct-level-change}
				M_{I'}^D
				\longrightarrow
				M_I^D\times_{C'_I}C'_{I'}
			\end{equation}
			is a finite \'{e}tale right torsor under
			\[
			G_{I',I}
			=
			\ker\!\left(\mathcal D_{I'}^\times
			\longrightarrow\mathcal D_I^\times\right).
			\]
			\item There is a surjective determinant morphism
			\[
			\wp_I:M_I^D\longrightarrow M_{I,C'_I}^{\mathcal O_C}
			\]
			to the restriction of the rank-one moduli scheme to \(C'_I\).  Every
			geometric fiber of \(\wp_I\) is
			irreducible, and
			\[
			\wp_I(zg)=\wp_I(z)\operatorname{Nrd}_I(g)
			\qquad(g\in\mathcal D_I^\times).
			\]
		\end{enumerate}
	\end{proposition}
	
	\begin{proof}
		(1) Smoothness and the relative dimension in part~(1) are proved in
		\cite[Thm. (4.1), p.~236]{LRS93}. For nonempty full level, there
		are no nontrivial automorphisms, and each piece with fixed
		\(\deg(\mathcal E_0)\) is a quasi-projective scheme of finite type;
		see \cite[Thm. (5.1) and its proof, p.~241]{LRS93}.
		The degree is locally constant in families, and the index shift
		changes it by \(2\). Thus, locally on the parameter scheme, each orbit
		has a unique representative of degree \(0\) or \(1\), and the quotient
		is represented by the disjoint union of these two fixed-degree pieces.
		In the division-algebra case, it is proper by
		\cite[Thm. (6.1) and Cor. (6.2), p.~246]{LRS93}.
		Being proper and quasi-projective, it is projective. This description
		of \(M_I^D\) is also given in \cite[\S3, p.~121]{Pap09}.
		
		\smallskip\noindent
		(2) The level-forgetting torsor first appears before the index-shift quotient.
		Indeed, \cite[(4.8), pp.~239--240]{LRS93} states that restriction of the
		level structure gives a right
		\(G_{I',I}\)-torsor
		\[
		Y:=\mathcal E\!ll_{I'}^{\mathcal D}
		\longrightarrow
		B:=\mathcal E\!ll_I^{\mathcal D}\times_{C'_I}C'_{I'}
		\]
		over \(C'_{I'}\). For \(a=0,1\), let \(Y_a\subset Y\) and \(B_a\subset B\) be
		the open and closed subschemes defined by
		\(\deg(\mathcal E_0)=a\). Forgetting part of the level structure
		does not change the underlying elliptic sheaf, so
		\[
		Y_a=Y\times_B B_a.
		\]
		Thus \(Y_a\to B_a\) is a right \(G_{I',I}\)-torsor by base change.
		
		As in part~(1), shifting the indices changes
		\(\deg(\mathcal E_0)\) by \(2\). Choosing the unique representative
		of degree \(0\) or \(1\) therefore identifies
		\[
		Y/\mathbf Z\simeq Y_0\amalg Y_1,
		\qquad
		B/\mathbf Z\simeq B_0\amalg B_1.
		\]
		Both the level-forgetting map and the \(G_{I',I}\)-action
		commute with index shifts and preserve the degree. Under these
		identifications, the quotient map is consequently the disjoint
		union of the two torsors \(Y_a\to B_a\).
		It is therefore a right \(G_{I',I}\)-torsor. Since this group
		is constant and finite, the map is finite \'{e}tale.
		This proves part~(2).
		
		\smallskip\noindent
		(3) The determinant morphism is likewise constructed on families before taking
		the quotient. Lafforgue's construction uses an \'{e}tale cover \(U\to C\)
		and an embedding of \(\mathcal O_U\)-algebras
		\[
		e:\mathcal D\otimes_{\mathcal O_C}\mathcal O_U
		\hookrightarrow\mathrm M_2(\mathcal O_U)
		\]
		which is an isomorphism at every generic point of \(U\)
		\cite[Chapitre~I, \S1(f), Lemme~3, pp.~26--27]{Laf97}.
		The embedding need not be an isomorphism above \(\Ram(D)\).
		
		Composing \(e\) with the ordinary matrix determinant gives a
		multiplicative morphism over \(U\). On overlaps, the generic embeddings
		differ by conjugation, so their determinants agree on the generic fibers.
		They therefore agree everywhere: the underlying vector bundle is reduced,
		and these generic fibers are dense. \'{E}tale descent gives a unique
		morphism of monoid schemes
		\[
		\operatorname{Nrd}:\mathcal D\longrightarrow\mathcal O_C
		\]
		extending the reduced norm of \(D\). In particular, this morphism is
		defined at the ramification points as well.
		
		To apply the reduced norm to \(\mathcal E_i\), we first explain why
		local \(\mathcal D\)-bases exist. Away from \(\Ram(D)\), we may split
		\(\mathcal D\) \'{e}tale-locally. By the explicit Morita equivalence in
		\cite[(9.5), p.~262]{LRS93}, \(\mathcal E_i\) then corresponds to a
		rank-two vector bundle; trivializing this vector bundle locally identifies
		\(\mathcal E_i\) with \(\mathcal D\) as a right \(\mathcal D\)-module.
		
		Now let \(v\) be a ramification point. Since \(v\) meets neither the
		characteristic nor the pole, the maps \(j_i\) and \(t_i\) restrict to
		isomorphisms on \(v\times S\). By
		\cite[Lem.~(2.6), pp.~225--226]{LRS93}, after an \'{e}tale base change
		on \(S\), the restriction \(\mathcal E_i|_{v\times S}\) is free of rank
		one over \(\mathcal D|_v\boxtimes\mathcal O_S\). Lift a generator
		locally to \(\mathcal E_i\). It defines a \(\mathcal D\)-linear map
		\[
		\mathcal D\boxtimes\mathcal O_S\longrightarrow\mathcal E_i
		\]
		whose restriction to \(v\times S\) is an isomorphism. Since both sides
		are vector bundles of rank four and the induced map on the fiber at the
		chosen point is an isomorphism, the map is an isomorphism on a
		neighborhood of that point \cite[Lem.~10.79.4(3)]{Stacks}.
		Thus \(\mathcal E_i\) is
		\'{e}tale-locally free of rank one as a right \(\mathcal D\)-module,
		including at the ramification points.
		
		For an \'{e}tale-locally free
		right \(\mathcal D\)-module \(\mathcal E\) of rank one on \(C\times S\),
		applying the reduced norm to the transition functions of local
		\(\mathcal D\)-bases defines the line bundle
		\(\det_{\mathcal D}(\mathcal E)\). The same construction assigns a morphism
		\(\det_{\mathcal D}(u)\) to each \(\mathcal D\)-linear map \(u\), and
		commutes with arbitrary base change. In the quaternionic case,
		\[
		\det\nolimits_{\mathcal D}(\mathcal E)^{\otimes2}
		\simeq\bigwedge\nolimits^{4}\mathcal E
		\otimes\operatorname{pr}_C^*
		\left(\bigwedge\nolimits^{4}\mathcal D\right)^\vee,
		\]
		where \(\operatorname{pr}_C:C\times S\to C\) is the projection.
		Indeed, left multiplication by \(g\) has determinant
		\(\operatorname{Nrd}(g)^2\) on the underlying rank-four module.
		We check that the determinants satisfy the rank-one conditions in
		Definition~\ref{def:moduli-scheme}. Let \(\Gamma\subset C\times S\) be the
		graph of the characteristic morphism. Both \(\Gamma\) and
		\(\{\infty\}\times S\) avoid the ramification points of \(D\), so
		\(\mathcal D\) is \'{e}tale-locally a matrix algebra near these sections.
		In local \(\mathcal D\)-bases, let \(B\) be the \(2\times2\) matrix
		representing \(j_i\) or \(t_i\). Left multiplication acts separately on
		the two columns, so the cokernel of the original map is the direct sum
		of two copies of \(\operatorname{coker}(B)\). Since this cokernel is a
		rank-two vector bundle on the corresponding section,
		\(\operatorname{coker}(B)\) is a line bundle there.
		
		The section is an effective Cartier divisor
		\cite[Lem.~44.3.1]{Stacks}. Work locally with coordinate ring \(R\),
		let \(z\) be an equation for the section, and trivialize
		\(\operatorname{coker}(B)\). We may choose a basis of the target so
		that the quotient map is
		\[
		R^2\longrightarrow R/(z),\qquad (a,b)\longmapsto b\bmod z.
		\]
		Its kernel is freely generated by \((1,0)\) and \((0,z)\). Since
		\(B\) identifies its source with this kernel, choosing their preimages
		as a basis of the source gives
		\[
		B=\begin{pmatrix}
			1&0\\
			0&z
		\end{pmatrix}.
		\]
		In these bases, the reduced norm is \(\det(B)=z\). Away from the
		section, the original map and its determinant are isomorphisms.
		
		Thus \(\det_{\mathcal D}(j_i)\) and \(\det_{\mathcal D}(t_i)\) are
		injective, with cokernels line bundles on \(\{\infty\}\times S\) and
		\(\Gamma\), respectively. In particular, \(\det_{\mathcal D}(j_i)\)
		identifies
		\[
		\det\nolimits_{\mathcal D}(\mathcal E_{i+1})
		\simeq\det\nolimits_{\mathcal D}(\mathcal E_i)(\infty),
		\]
		which is the required rank-one periodicity. Functoriality preserves
		the commutative squares. The standard generator \(1\) of
		\(\mathcal D|_I\boxtimes\mathcal O_S\) canonically trivializes its
		determinant. Thus the full level structure \(\iota\) induces an
		isomorphism
		\[
		\det\nolimits_{\mathcal D}(\iota):
		\mathcal O_{I\times S}\xrightarrow{\sim}
		\det\nolimits_{\mathcal D}(\mathcal E_I).
		\]
		Since taking determinants commutes with Frobenius pullback, the
		identity \(t_I\circ\tau(\iota)=\iota\) gives
		\[
		\det\nolimits_{\mathcal D}(t_I)\circ
		\tau\!\bigl(\det\nolimits_{\mathcal D}(\iota)\bigr)
		=\det\nolimits_{\mathcal D}(\iota).
		\]
		Hence \(\det_{\mathcal D}(\iota)\) is a full level-\(I\) structure on
		the resulting rank-one elliptic sheaf; see
		\cite[Chapitre~I, \S1(f), pp.~28--29]{Laf97}.
		We therefore obtain a rank-one elliptic sheaf
		with full level structure by sending
		\[
		\bigl((\mathcal E_i,j_i,t_i)_i,\iota\bigr)
		\longmapsto
		\bigl((\det\nolimits_{\mathcal D}\mathcal E_i,
		\det\nolimits_{\mathcal D}j_i,\det\nolimits_{\mathcal D}t_i)_i,
		\det\nolimits_{\mathcal D}\iota\bigr).
		\]
		This construction commutes with base change, so it gives a morphism of
		the fppf moduli sheaves before quotienting by \(\mathbf Z\). Reindexing
		an elliptic sheaf reindexes its determinant by the same amount and
		transports the induced level structure accordingly. After composing
		with the quotient map on the rank-one side, the resulting morphism is
		therefore invariant under index shifts. It factors uniquely through
		the quotient on the source by the universal property of quotient sheaves
		\cite[\S39.20]{Stacks}. Since these quotients are represented by the
		schemes in Definition~\ref{def:moduli-scheme} and the characteristic
		morphism is unchanged, we obtain
		\[
		\wp_I:M_I^D\longrightarrow M_{I,C'_I}^{\mathcal O_C}.
		\]
		It agrees with the morphism used in \cite[\S3, p.~122]{Pap09}.
		For \(g\in\mathcal D_I^\times\), we have
		\(\det_{\mathcal D}(L_g)=\operatorname{Nrd}_I(g)\).
		Functoriality and the right action \(\iota\cdot g=\iota\circ L_g\)
		therefore give
		\[
		\det\nolimits_{\mathcal D}(\iota\circ L_g)
		=\det\nolimits_{\mathcal D}(\iota)\cdot\operatorname{Nrd}_I(g),
		\qquad
		\wp_I(zg)=\wp_I(z)\operatorname{Nrd}_I(g).
		\]
		Thus \(\wp_I\) is equivariant for the reduced norm, as in
		\cite[Chapitre~I, \S1(f), p.~29]{Laf97}, and so are its induced maps
		on geometric connected components.
		
		To see that this morphism is surjective over every characteristic, we first
		consider the geometric generic fiber. Let
		\(\bar\eta=\operatorname{Spec}\overline{\mathbf F_p(T)}\) be a geometric
		generic point of \(C'_I\). Base change to \(\bar\eta\) gives the morphism
		\[
		\wp_{I,\bar\eta}\colon
		M_{I,\bar\eta}^{D}\longrightarrow M_{I,\bar\eta}^{\mathcal O_C},
		\]
		where the source and target denote the corresponding geometric generic
		fibers. The geometric points of \(M_{I,\bar\eta}^{\mathcal O_C}\)
		form a nonempty finite set on which \((A/I)^\times\) acts
		transitively
		\cite[proof of Prop.~3.1, p.~122]{Pap09}.
		For the underlying class-field-theoretic description, see
		\cite[\S8, Thm.~1, p.~585]{Dri74}.
		The number of components of the geometric generic fiber of \(M_I^D\)
		is the same nonzero ray class number
		\cite[Cor.~6.2 and its proof, p.~130]{Pap09}, so that fiber is
		nonempty. Its image under
		\(\wp_{I,\bar\eta}\) is therefore nonempty and, by equivariance and the
		surjectivity of
		\(\operatorname{Nrd}_I:\mathcal D_I^\times\to(A/I)^\times\), contains the
		entire orbit of any one of its points.  Hence \(\wp_{I,\bar\eta}\) is
		surjective.
		
		The rank-one target is finite \'{e}tale over \(C'_I\). Indeed, the
		same moduli results, applied with \(\mathcal B=\mathcal O_C\) and
		\(d_{\mathcal B}=1\), make it smooth and proper of relative dimension
		zero; see \cite[Thms. (4.1), (5.1) and Cor. (6.2), pp.~236, 241, 246]{LRS93}.
		Consequently every connected
		component of the target meets the geometric generic fiber and is the closure
		of its generic points. By part~(1), the source is proper over \(C'_I\),
		while the target is finite, hence separated, over \(C'_I\).
		Therefore \(\wp_I\) is proper
		\cite[Lem.~29.42.7, Tag~01W6]{Stacks}. Its image is therefore
		closed; generic surjectivity makes this image contain the generic points of
		every target component, and hence their closures.  Thus \(\wp_I\) is
		surjective.
		
		We now prove that every geometric fiber is irreducible, following
		\cite[Prop.\ 3.1, p.~122]{Pap09}. The preceding component count shows
		that \(M_{I,\bar\eta}^{D}\) has exactly as many connected components as
		\(M_{I,\bar\eta}^{\mathcal O_C}\) has points. Since
		\(\wp_{I,\bar\eta}\) is surjective and each connected component maps to a
		single point, every point of the target has exactly one connected component
		above it. Thus the fibers of \(\wp_{I,\bar\eta}\) are connected.
		
		The morphism \(\wp_I\) is smooth as well as proper. Indeed,
		\'{e}tale-locally on \(C'_I\), its finite \'{e}tale target is a disjoint
		union of copies of the base. The inverse image of each copy is an open and
		closed subscheme of the smooth source, so the restricted morphism is
		smooth. For a smooth proper morphism, the number of connected components
		of the geometric fibers is locally constant; see
		\cite[Lem.~37.53.8]{Stacks}. This number is one at every geometric generic
		point of the target, by the preceding paragraph. Since every connected
		component of the target meets the generic fiber, it is one everywhere.
		
		Finally, each geometric fiber is a smooth connected curve. Distinct
		irreducible components of a smooth curve cannot meet, because its local
		rings are regular local rings and hence domains
		\cite[Lem.~10.106.2]{Stacks}. Connectedness therefore
		implies irreducibility.
	\end{proof}
	
	Proposition~\ref{prop:modular-input} reduces the component calculation to the
	rank-one moduli space.  We therefore compute its geometric points and the
	action of Frobenius explicitly.
	
	\begin{lemma}
		\label{lem:rank-one}
		Let \(I\) be a nonempty finite level whose support does not contain the
		point \((T)\).  The
		geometric points of the rank-one fiber
		\[
		M_{I,(T)}^{\mathcal O_C}
		\]
		form a torsor under
		\[
		R_I=(A/I)^\times/\Fp^\times.
		\]
		Arithmetic Frobenius acts on this torsor by right translation by the
		class of \((T\bmod I)^{-1}\).
	\end{lemma}
	
	\begin{proof}
		Put \(L=\Fbar\) and \(C_L=C\times_{\Fp}L\). We first describe the
		underlying rank-one elliptic sheaf. After shifting the indices, we may
		assume that \(\deg\mathcal E_0=0\). Since every degree-zero line bundle
		on \(C_L=\mathbf P^1_L\) is trivial \cite[Lem.~31.29.5]{Stacks},
		periodicity allows us to write
		\[
		\mathcal E_i=\mathcal O_{C_L}(i\infty),
		\]
		with \(j_i\) the natural inclusions. The cokernel of \(t_i\) is supported
		at \((T)\), so \(t_i\) is multiplication by \(cT\) for some
		\(c\in L^\times\). Commutativity with the \(j_i\) makes \(c\)
		independent of \(i\). Since \(L\) is algebraically closed, choose
		\(u\in L^\times\) with \(u^{p-1}=c\). Multiplication by \(u\) on each
		\(\mathcal E_i\) commutes with the \(j_i\) and gives an isomorphism to
		the elliptic sheaf with \(t_i=T\), since \(u(cT)=Tu^p\).
		
		Thus, up to isomorphism and index shift, there is a single underlying
		elliptic sheaf, represented by
		\[
		\mathcal E_i=\mathcal O_{C_L}(i\infty),
		\qquad t_i=T.
		\]
		This model is defined over \(\Fp\). Its automorphisms are common scalar
		multiplications by \(u\in L^\times\) satisfying \(u^p=u\), and hence
		form the group \(\Fp^\times\).
		
		Put \(R=A/I\) and define
		\[
		\sigma
		:=\operatorname{id}_R\otimes\Frob
		=\operatorname{id}_R\otimes(a\mapsto a^p)
		\]
		on \(R\otimes_{\Fp}L\). In particular, \(\sigma\) fixes \(T\bmod I\).
		Restricting the normalized elliptic sheaf to \(I\), a full level
		structure is multiplication by a unit
		\[
		b=\iota(1)\in(R\otimes_{\Fp}L)^\times.
		\]
		The compatibility condition \(t_I\circ\tau(\iota)=\iota\) in
		Definition~\ref{def:moduli-scheme} is therefore
		\[
		(T\bmod I)\sigma(b)=b.
		\]
		
		Such level structures exist by
		\cite[Lem. (2.6) and Def. (2.7), pp.~225--226]{LRS93}. If \(b\)
		and \(b'\) define two of them, then
		\[
		\sigma(b'/b)=b'/b.
		\]
		Expanding in an \(\Fp\)-basis of \(R\), we see that the
		\(\sigma\)-fixed elements of \(R\otimes_{\Fp}L\) are exactly those of
		\(R\), since their coefficients satisfy \(a^p=a\). Both \(b'/b\) and
		its inverse are fixed, so \(b'/b\in R^\times\).
		Conversely, multiplication by any element of
		\(R^\times\) preserves the compatibility condition. The full level
		structures consequently form a right \(R^\times\)-torsor. Quotienting
		by the automorphism group \(\Fp^\times\) gives the asserted
		\(R_I\)-torsor.
		
		Since the normalized model is defined over \(\Fp\), arithmetic
		Frobenius sends \(b\) to \(\sigma(b)\). The compatibility condition gives
		\[
		\sigma(b)=(T\bmod I)^{-1}b.
		\]
		With the right action \(\iota\cdot a=\iota\circ L_a\), this says
		precisely that
		\[
		\Frob(\iota)=\iota\cdot(T\bmod I)^{-1}.
		\]
		The same formula descends to the quotient by \(\Fp^\times\).
	\end{proof}
	
	The irreducible fibers of the determinant morphism now carry the rank-one
	component calculation back to the quaternionic fiber.  This is the form that
	will be used to compute the monodromy of the level covers in Section~5.
	
	\begin{corollary}
		\label{cor:modular-components}
		Assume \((T)\notin\Ram(D)\).  For every nonempty finite level \(I\) with
		\((T)\nmid I\), the geometric component set
		\[
		\pi_0\!\left(M_{I,(T)}^D\otimes_{\Fp}\Fbar\right)
		\]
		is a torsor under \(R_I\).  The action of \(\mathcal D_I^\times\) on
		components is translation through
		\[
		\operatorname{Nrd}_I:\mathcal D_I^\times\longrightarrow(A/I)^\times
		\longrightarrow R_I,
		\]
		and arithmetic Frobenius acts by the class of \((T\bmod I)^{-1}\).
	\end{corollary}
	
	\begin{proof}
		Since \((T)\notin\Ram(D)\) and \((T)\nmid I\), the point \((T)\) lies in
		\(C'_I\).  Base-changing \(\wp_I\) to \((T)\) therefore gives a morphism
		from \(M_{I,(T)}^D\) to \(M_{I,(T)}^{\mathcal O_C}\).  Its rank-one target
		is finite
		\'{e}tale at the prime-to-\((T)\) level \(I\), and
		Lemma~\ref{lem:rank-one} describes its geometric points.
		After base change to \(\Fbar\), the target is a finite disjoint union
		of points. By Proposition~\ref{prop:modular-input}(3), the inverse image
		of each point is nonempty and irreducible. These inverse images are open
		and closed in \(M_{I,(T)}^D\otimes_{\Fp}\Fbar\) and form a disjoint
		covering, so they are precisely its connected components.
		The equivariance and Frobenius formula now follow from
		Proposition~\ref{prop:modular-input}(3) and
		Lemma~\ref{lem:rank-one}.
	\end{proof}
	
	\section{A fixed curve and its auxiliary level torsors}
	
	We now choose a single quaternion algebra \(D\) and a single fiber that will
	work for every \(r\geq4\).  The algebra will be split at \((T)\), \((T-1)\),
	and \(\infty\), while all its ramification points will have degree at most
	three.  The first condition makes the fiber and levels used below available;
	the degree bound will make the choice of auxiliary primes uniform in \(r\).
	
	Choose two finite closed points
	\(\mathfrak p_1,\mathfrak p_2\) according to the table
	\[
	\begin{array}{c|c}
		p & (\mathfrak p_1,\mathfrak p_2) \\ \hline
		2 & \bigl((T^2+T+1),(T^3+T+1)\bigr),\\
		3 & \bigl((T-2),(T^2+1)\bigr),\\
		p\ge5 & \bigl((T-2),(T-3)\bigr).
	\end{array}
	\]
	In each row the two displayed polynomials are irreducible, define distinct
	points of degree at most three, and avoid \((T),(T-1),\infty\).  Assign
	local invariant \(1/2\) at these two points and \(0\) elsewhere.  Since
	\(1/2+1/2=0\) in \(\mathbf Q/\mathbf Z\), the
	Albert--Brauer--Hasse--Noether exact sequence gives a unique Brauer class
	with these invariants.  Its local indices are \(2\) at
	\(\mathfrak p_1,\mathfrak p_2\) and \(1\) elsewhere, so its global index is
	\(2\).  Its division representative is therefore a quaternion algebra
	\(D\) over \(\Fp(T)\), ramified exactly at \(\mathfrak p_1,\mathfrak p_2\) and split
	at \((T),(T-1),\infty\); see \cite[Rem. (32.12)(ii) and exact sequence
	(32.13), pp.~277--278, and Thm. (32.19), p.~280]{Rei03}.  Fix this \(D\)
	and a maximal order
	\(\mathcal D\) from now on.
	
	The level \((T-1)\) now has its intended effect: the corresponding rank-one
	component group is trivial, so the fiber over \((T)\) is geometrically
	connected.  This gives the curve on which the representations required by
	Theorem~\ref{thm:main} will first be constructed.
	
	\begin{proposition}\label{prop:base}
		The curve
		\[
		X:=M_{(T-1),(T)}^D
		\]
		is smooth, projective, and geometrically connected over \(\Fp\).
	\end{proposition}
	
	\begin{proof}
		Part (1) of Proposition~\ref{prop:modular-input} gives smoothness,
		projectivity, and pure
		dimension one.  Moreover,
		\[
		R_{(T-1)}=(A/(T-1))^\times/\Fp^\times
		=\Fp^\times/\Fp^\times=1.
		\]
		Corollary~\ref{cor:modular-components} therefore says that
		\(X_{\Fbar}\) has one geometrically irreducible component.
	\end{proof}
	
	To prove Theorem~\ref{thm:main}, we first construct two-dimensional
	representations on the fixed curve \(X\). These arise from the monodromy
	of finite \'{e}tale torsors obtained by adding an auxiliary prime to the
	level and then forgetting it. Theorem~\ref{thm:level-components} will
	determine their geometric images and determinants. In Section~6, a
	suitable choice of the auxiliary prime will then allow us to remove the
	determinant by a scalar twist and obtain the required representations.
	
	Let \(\ell\subset A\) be a nonzero prime ideal disjoint from
	\[
	\{(T),(T-1)\}\cup\Ram(D).
	\]
	Put
	\[
	k_\ell=A/\ell.
	\]
	Proposition~\ref{prop:modular-input}(2), applied to the levels \((T-1)\)
	and \((T-1)\ell\) and then to the fiber over \((T)\), gives a finite
	\'{e}tale right \(G_\ell\)-torsor
	\begin{equation}\label{eq:level-torsor}
		Z_\ell:=M_{(T-1)\ell,(T)}^D
		\longrightarrow M_{(T-1),(T)}^D=X,
	\end{equation}
	where
	\[
	G_\ell
	:=
	\ker\!\left(\mathcal D_{(T-1)\ell}^{\times}
	\longrightarrow\mathcal D_{(T-1)}^\times\right)
	\cong\mathcal D_\ell^\times.
	\]
	Indeed, the Chinese remainder theorem identifies the restriction
	homomorphism with projection onto the \((T-1)\)-factor. Since \(\ell\)
	lies outside \(\Ram(D)\), choose a splitting
	\[
	\mathcal D_\ell\simeq\mathrm{M}_2(k_\ell).
	\]
	This identifies \(G_\ell\) with \(\GL_2(k_\ell)\) and the reduced norm
	with the determinant.
	Choose a geometric base point \(\bar x\to X\) and a point
	\(\bar z\in Z_{\ell,\bar x}\), and
	let
	\[
	\theta_\ell:\pi_1^{\et}(X)\longrightarrow\GL_2(k_\ell)
	\]
	be the resulting monodromy homomorphism.
	
	The geometry of Section~4 now determines the two facts needed in the proof
	of Theorem~\ref{thm:main}: the geometric image of \(\theta_\ell\) and the
	determinant of arithmetic Frobenius.
	
	\begin{theorem}\label{thm:level-components}
		For the torsor \eqref{eq:level-torsor}:
		\begin{enumerate}
			\item the geometric component set of \(Z_\ell\) is a torsor under
			\(k_\ell^\times\);
			\item the right \(\GL_2(k_\ell)\)-action on \(Z_\ell\) induces
			translation on the component torsor through the determinant;
			\item
			\begin{equation}\label{eq:geo-image}
				\theta_\ell\bigl(\pi_1^{\et}(X_{\Fbar})\bigr)
				=\SL_2(k_\ell);
			\end{equation}
			\item The character \(\delta_\ell=\det\theta_\ell\) factors through
			\(\Gal(\Fbar/\Fp)\). Denoting the resulting character by the same
			symbol, we have
			\begin{equation}\label{eq:frob-det}
				\delta_\ell(\Frob)=(T\bmod\ell)^{-1}.
			\end{equation}
		\end{enumerate}
	\end{theorem}
	
	\begin{proof}
		(1) We first determine the component set.  Corollary~\ref{cor:modular-components}
		makes it a torsor
		under
		\[
		R_{(T-1)\ell}
		=
		\bigl(A/((T-1)\ell)\bigr)^\times/\Fp^\times.
		\]
		The Chinese remainder theorem gives a canonical isomorphism of acting
		groups
		\begin{equation}\label{eq:crt-ray}
			R_{(T-1)\ell}
			\cong
			\frac{\Fp^\times\times k_\ell^\times}
			{\Delta\Fp^\times}
			\xrightarrow{\ \sim\ } k_\ell^\times,
			\qquad
			[(a,b)]\longmapsto a^{-1}b.
		\end{equation}
		This identifies the component set with a \(k_\ell^\times\)-torsor;
		it does
		not choose an origin in that torsor.
		
		\smallskip\noindent
		(2) To determine the geometric image in (3), we first need to know
		which elements of \(\GL_2(k_\ell)\) stabilize a geometric component.
		Under the Chinese remainder isomorphism and the chosen splitting,
		an element \(g\in\GL_2(k_\ell)\) of the torsor group corresponds to
		\[
		(1,g)\in
		\mathcal D_{(T-1)}^\times\times\GL_2(k_\ell).
		\]
		The reduced norm is computed separately on the two factors and
		agrees with the determinant on the second. Thus, under these
		identifications,
		\[
		\operatorname{Nrd}_{(T-1)\ell}(1,g)
		=(1,\det g)\in\Fp^\times\times k_\ell^\times.
		\]
		By Proposition~\ref{prop:modular-input}(3), the induced action on
		the component torsor is therefore translation by the class
		\([(1,\det g)]\). Under \eqref{eq:crt-ray}, this class maps to
		\(\det g\), proving (2).
		
		\smallskip\noindent
		(3) Because \(k_\ell^\times\) acts freely on the
		component set, \(g\) fixes a component if and only if \(\det g=1\).
		Since the determinant is surjective,
		Corollary~\ref{cor:stabilizer} gives
		\[
		\theta_\ell\bigl(\pi_1^{\et}(X_{\Fbar})\bigr)
		=
		\ker\!\left(\det:\GL_2(k_\ell)\to k_\ell^\times\right)
		=
		\SL_2(k_\ell).
		\]
		
		\smallskip\noindent
		(4) To remove the determinant without changing the geometric image, we will
		use a character of the absolute Galois group of the constant field
		\(\Fp\), as in Lemma~\ref{lem:twist}. By (3), \(\delta_\ell\) is trivial on
		\(\pi_1^{\et}(X_{\Fbar})\), so the fundamental-group exact sequence
		\eqref{eq:fundamental-exact-sequence} shows that it factors through
		\(\Gal(\Fbar/\Fp)\). Its value at arithmetic Frobenius can therefore be
		computed using any lift to \(\pi_1^{\et}(X)\): two such lifts differ by an
		element of the geometric fundamental group.
		
		We now compute this value on the rank-one moduli space, where the
		Frobenius action is given explicitly by Lemma~\ref{lem:rank-one}. Write
		\[
		\wp_{(T-1)\ell,(T)}:
		Z_\ell\longrightarrow M^{\mathcal O_C}_{(T-1)\ell,(T)}
		\]
		for the base change of the determinant morphism to the fiber over
		\((T)\). Its target is finite \'{e}tale over \(\Fp\). Together with the
		projection to \(X\), it gives a morphism of finite \'{e}tale
		\(X\)-schemes
		\[
		Z_\ell\longrightarrow
		X\times_{\Fp}M^{\mathcal O_C}_{(T-1)\ell,(T)}.
		\]
		
		Let \(d:=\wp_{(T-1)\ell,(T)}(\bar z)\), viewed as a point of the
		geometric fiber of \(M^{\mathcal O_C}_{(T-1)\ell,(T)}\) determined by
		\(\bar x\). Choose a lift \(\gamma\in\pi_1^{\et}(X)\) of arithmetic
		Frobenius. The induced map on geometric fibers is
		\(\pi_1^{\et}(X)\)-equivariant. Since the target cover is pulled back from
		\(\operatorname{Spec}\Fp\), its fundamental-group action is induced by
		the natural Galois action on \(M^{\mathcal O_C}_{(T-1)\ell,(T)}\); see
		\cite[Thm.~58.6.2(1), (3)]{Stacks}.
		
		The monodromy convention of Lemma~\ref{lem:components} and determinant
		equivariance therefore give
		\[
		\begin{aligned}
			\Frob(d)
			&=\wp_{(T-1)\ell,(T)}\bigl(\gamma(\bar z)\bigr)\\
			&=\wp_{(T-1)\ell,(T)}\bigl(\bar z\,\theta_\ell(\gamma)\bigr)\\
			&=d\cdot\delta_\ell(\gamma),
		\end{aligned}
		\]
		where the last action uses the identification \eqref{eq:crt-ray}.
		
		On the other hand, since \(T\equiv1\bmod(T-1)\), the Chinese remainder
		isomorphism sends \((T\bmod((T-1)\ell))^{-1}\) to
		\[
		\bigl(1,(T\bmod\ell)^{-1}\bigr).
		\]
		Lemma~\ref{lem:rank-one} and the identification \eqref{eq:crt-ray}
		therefore give
		\[
		\Frob(d)=d\cdot(T\bmod\ell)^{-1}.
		\]
		Comparing the two expressions for \(\Frob(d)\), and using the freeness
		of the \(k_\ell^\times\)-action, we obtain
		\[
		\delta_\ell(\Frob)
		=\delta_\ell(\gamma)
		=(T\bmod\ell)^{-1},
		\]
		as required.
	\end{proof}
	
	Theorem~\ref{thm:level-components} leaves one choice to make.  To apply the
	scalar twist of Lemma~\ref{lem:twist} for every \(r\geq4\), we need a
	degree-\(r\) prime \(\ell\) for which \(T\bmod\ell\) is a square.  The
	degree bound on the fixed ramification set makes this possible uniformly.
	
	\begin{lemma}\label{lem:square-primes}
		Every point of
		\[
		\{(T),(T-1)\}\cup\Ram(D)
		\]
		has degree at most \(3\). Moreover, for every integer \(r\geq4\),
		there is a prime ideal \(\ell_r\subset A\), disjoint from this set, such
		that
		\[
		\deg(\ell_r)=r,\qquad
		T\bmod\ell_r\in
		\bigl(\mathbf F_{p^r}^{\times}\bigr)^2.
		\]
	\end{lemma}
	
	\begin{proof}
		The points \((T)\) and \((T-1)\) have degree \(1\), while the two points
		in \(\Ram(D)\) have degree at most \(3\) by the choice of \(D\).
		
		Fix \(r\geq4\). Choose a generator
		\(\beta\in\mathbf F_{p^r}^{\times}\), and put
		\(\alpha=\beta^2\).  The element \(\alpha\) is already a square; we only
		need to show that it still generates \(\mathbf F_{p^r}\) over \(\Fp\).  Its
		order is
		\[
		\operatorname{ord}(\alpha)
		=\frac{p^r-1}{\gcd(2,p^r-1)}
		\geq\frac{p^r-1}{2}.
		\]
		
		If \(\alpha\) belonged to a proper subfield
		\(\mathbf F_{p^d}\subsetneq\mathbf F_{p^r}\), then \(d\mid r\), hence
		\(d\leq r/2\), and therefore
		\[
		\operatorname{ord}(\alpha)
		\leq p^d-1
		\leq p^{r/2}-1
		<\frac{p^r-1}{2},
		\]
		contradicting the preceding inequality. Thus
		\(\Fp(\alpha)=\mathbf F_{p^r}\).
		
		Let \(f_r\) be the minimal polynomial of \(\alpha\) over \(\Fp\) and set
		\(\ell_r=(f_r)\). Then \(\deg(\ell_r)=r\), and under the isomorphism
		\[
		A/\ell_r\xrightarrow{\ \sim\ }\mathbf F_{p^r},
		\qquad T\longmapsto\alpha,
		\]
		the class \(T\bmod\ell_r=\beta^2\) is a square. Since \(\ell_r\) has
		degree \(r\geq4\), whereas every forbidden point has degree at most
		\(3\), it is disjoint from \(\{(T),(T-1)\}\cup\Ram(D)\).
	\end{proof}
	
	\begin{remark}\label{rem:equal-characteristic}
		The role of equal characteristic is quite concrete.  Write
		\(p'=\operatorname{char}K\), and let \(p\) be the characteristic of the
		coefficient field.  When \(p=p'\), a degree-\(r\) prime
		\(\ell\subset\Fp[T]\) has residue field
		\[
		k_\ell=A/\ell\cong\mathbf F_{p^r},
		\]
		which embeds in \(\Fbar\).  The natural two-dimensional representation of
		the level group can therefore be used directly as a mod-\(p\)
		representation.
		
		The same coincidence of characteristics appears in the Frobenius
		calculation. In Lemma~\ref{lem:rank-one}, arithmetic Frobenius acts on
		the coefficients of a level structure by \(a\mapsto a^p\). Compatibility
		with the map \(t_I\), which is multiplication by \(T\) in the normalized
		rank-one elliptic sheaf, gives
		\[
		\Frob(\iota)=\iota\cdot(T\bmod I)^{-1}.
		\]
		Theorem~\ref{thm:level-components}(4) then gives
		\[
		\delta_\ell(\Frob)=(T\bmod\ell)^{-1}.
		\]
		From this point on, the argument no longer uses the equality \(p=p'\).
		Identifying the geometric monodromy with the stabilizer
		\(\SL_2(k_\ell)\), choosing \(\ell\) so that \(T\bmod\ell\) is a square,
		and removing the determinant by a scalar twist are finite-field and
		group-theoretic arguments.
		
		This also explains why our construction does not conflict with the result
		for unequal characteristics recalled in the introduction.  If \(p'\neq p\),
		the corresponding level construction over \(\mathbf F_{p'}\) has residue
		fields of characteristic \(p'\), and its natural representation takes
		values in \(\GL_2(\overline{\mathbf F}_{p'})\).  It does not produce the
		mod-\(p\) representations considered in
		Conjecture~\ref{MTfinitenessconj}.  For those representations, the
		cross-characteristic theorem \cite[Thm.~1.3, p.~2]{Luo26a} proves finiteness in
		arbitrary dimension when \(p\neq2\), and in dimension \(2\) when \(p=2\).
		Thus the essential point in equal characteristic is that the same prime
		governs both the level coefficients and the Frobenius action; the
		component-stabilizer argument and the scalar twist themselves do not depend
		on this coincidence.
	\end{remark}
	
	\section{Proof of Theorem~\ref{thm:main}}
	
	The proof has two stages.  We first remain on the fixed curve \(X\):
	Theorem~\ref{thm:level-components} determines the geometric image and the
	Frobenius determinant, Lemma~\ref{lem:square-primes} makes that determinant a
	square, and Lemmas~\ref{lem:twist} and \ref{lem:natural-irred} produce the
	required absolutely irreducible representations.  We then apply
	Proposition~\ref{prop:monodromy-towers} to carry all these representations
	simultaneously to finite \'{e}tale covers of strictly increasing genus.
	
	\begin{proof}[Proof of Theorem~\ref{thm:main}]
		Fix a prime \(p\), and let
		\[
		X=M_{(T-1),(T)}^D
		\]
		be the curve constructed in Proposition~\ref{prop:base}.  We first construct
		the required representations on this fixed curve.  Given \(r\geq 4\), choose
		\(\ell_r\) as in Lemma~\ref{lem:square-primes}, put
		\[
		k_r=A/\ell_r\cong\mathbf F_{p^r},
		\]
		and fix an \(\Fp\)-embedding \(k_r\hookrightarrow\Fbar\).
		
		The finite \'{e}tale level torsor \(Z_{\ell_r}\to X\) defines a
		representation
		\[
		\theta_r:\pi_1^{\et}(X)\longrightarrow\GL_2(k_r).
		\]
		Theorem~\ref{thm:level-components} gives
		\[
		\theta_r\bigl(\pi_1^{\et}(X_{\Fbar})\bigr)=\SL_2(k_r)
		\]
		and, if \(\delta_r=\det\theta_r\),
		\[
		\delta_r(\Frob)=(T\bmod\ell_r)^{-1}.
		\]
		The choice of \(\ell_r\) ensures that \(T\bmod\ell_r\) is a square in
		\(k_r^\times\), and hence so is its inverse.
		Lemma~\ref{lem:twist} therefore supplies a continuous
		character
		\[
		\chi_r:\pi_1^{\et}(X)\longrightarrow k_r^\times,
		\]
		for which the representation
		\[
		\rho_r(\gamma)=\chi_r(\gamma)I_2\theta_r(\gamma)
		\qquad(\gamma\in\pi_1^{\et}(X)),
		\]
		has trivial determinant.  The character \(\chi_r\) is trivial on
		\(\pi_1^{\et}(X_{\Fbar})\), so the twist does not change the geometric image:
		\[
		\rho_r\bigl(\pi_1^{\et}(X_{\Fbar})\bigr)=\SL_2(k_r).
		\]
		On the other hand, \(\det\rho_r=1\), so the arithmetic image is contained in
		\(\SL_2(k_r)\).  Since it contains the geometric image, we obtain
		\[
		\rho_r\bigl(\pi_1^{\et}(X_{\Fbar})\bigr)
		=
		\rho_r\bigl(\pi_1^{\et}(X)\bigr)
		=
		\SL_2(k_r).
		\]
		Using the chosen embedding \(k_r\hookrightarrow\Fbar\), we regard \(\rho_r\)
		as a continuous geometric representation into \(\GL_2(\Fbar)\).  It is absolutely
		irreducible by Lemma~\ref{lem:natural-irred}.
		
		We now pass from the fixed curve \(X\) to a single tower that works for
		every \(r\geq4\).  Two facts are needed: the inequality \(g(X)\geq2\)
		ensures that the curves in the tower are geometrically distinct, while
		the perfectness of \(\SL_2(\mathbf F_{p^r})\) ensures that their monodromy
		groups remain unchanged.
		
		First, taking \(r=4\) above gives the nonabelian quotient
		\(\SL_2(\mathbf F_{p^4})\) of \(\pi_1^{\et}(X_{\Fbar})\).  Over an algebraically
		closed field, a smooth projective curve of genus zero has trivial
		\'{e}tale fundamental group: the Riemann--Hurwitz formula forces every
		connected finite \'{e}tale cover to have degree one
		\cite[Lem.~53.12.2]{Stacks}. A curve of genus one becomes an elliptic
		curve after a base point is chosen. Every connected finite \'{e}tale
		cover then comes from a separable isogeny, whose deck transformations
		are translations by points of its kernel. Its \'{e}tale fundamental
		group is therefore abelian
		\cite[\S18, Serre--Lang theorem and the ensuing discussion,
		pp.~155--158]{Mumford08}. Neither case admits the preceding quotient. Hence
		\(g(X)\geq2\).
		
		We next recall why the groups that occur here are perfect.  More
		generally, let \(q>3\), choose \(a\in\mathbf F_q^\times\) with
		\(a^2\neq1\), and write
		\[
		d_a=
		\begin{pmatrix}
			a&0\\
			0&a^{-1}
		\end{pmatrix},
		\qquad
		u_+(x)=
		\begin{pmatrix}
			1&x\\
			0&1
		\end{pmatrix},
		\qquad
		u_-(x)=
		\begin{pmatrix}
			1&0\\
			x&1
		\end{pmatrix}.
		\]
		With the convention \([g,h]=ghg^{-1}h^{-1}\), direct calculation gives
		\[
		[d_a,u_+(x)]
		=
		u_+\bigl((a^2-1)x\bigr),
		\qquad
		[d_a,u_-(x)]
		=
		u_-\bigl((a^{-2}-1)x\bigr).
		\]
		Both coefficients are nonzero.  Hence every upper and lower unipotent
		matrix is a commutator.  Since these matrices generate
		\(\SL_2(\mathbf F_q)\) by elementary row operations
		\cite[\S3.3.2, pp.~45--46]{Wilson09}, this group is perfect whenever \(q>3\), and in
		particular when \(q=p^r\) with \(r\geq4\).
		
		Choose a rational prime \(\lambda\) different from \(p\), and let
		\[
		X_j\longrightarrow X,\qquad j\geq1,
		\]
		be the tower supplied by Proposition~\ref{prop:monodromy-towers}.
		Since \(g(X)\geq2\), part~(2) of that proposition shows that the curves
		\(X_j\) are pairwise geometrically nonisomorphic.
		
		After choosing compatible geometric base points, define
		\[
		\rho_{j,r}
		:=
		\rho_r|_{\pi_1^{\et}(X_j)}.
		\]
		For every \(r\geq4\), the group \(\SL_2(\mathbf F_{p^r})\) is perfect
		and the restriction of \(\rho_r\) to \(\pi_1^{\et}(X_{\Fbar})\) is surjective.
		Part~(3) of Proposition~\ref{prop:monodromy-towers} therefore gives
		\[
		\rho_{j,r}\bigl(\pi_1^{\et}((X_j)_{\Fbar})\bigr)
		=
		\rho_{j,r}\bigl(\pi_1^{\et}(X_j)\bigr)
		=
		\SL_2(\mathbf F_{p^r}).
		\]
		The tower in Proposition~\ref{prop:monodromy-towers} depends only on
		\(X\) and \(\lambda\), so the same curves \(X_j\) work for every
		\(r\geq4\).  Finally, each \(\rho_{j,r}\) is geometric by the displayed
		equality and absolutely irreducible by Lemma~\ref{lem:natural-irred}.
		
		For any fixed $j$, the representations $\rho_{j,r}$ are pairwise nonisomorphic. Indeed, isomorphic
		representations have conjugate images and hence images of the same order, but
		\[
		\#\SL_2(\mathbf F_{p^r})=p^r\bigl(p^{2r}-1\bigr)
		\]
		is strictly increasing with \(r\). It follows that Conjecture~\ref{ass:naive} and hence
		Conjecture~\ref{MTfinitenessconj} fail for \(n=2\). This completes our proof.
	\end{proof}
	
	\section{Proof of Corollary~\ref{cor:infinite-product} and Corollary~\ref{cor:nonliftable-family}}
	\label{sec:infinite-product}
	
	To prove the first assertion of Corollary~\ref{cor:infinite-product},
	we combine the projectivizations of the representations in
	Theorem~\ref{thm:main} into a surjection onto
	\[
	\prod_{r\geq4}\PSL_2(\mathbf F_{p^r}).
	\]
	Surjectivity onto each factor is not sufficient: for a nontrivial
	finite group \(G\), the diagonal subgroup of \(G\times G\) projects
	onto both factors but is a proper subgroup.  The following lemma
	gives the required surjectivity when the factors are pairwise
	nonisomorphic finite nonabelian simple groups.
	
	\begin{lemma}
		\label{lem:independent-simple-quotients}
		Let \(\Gamma\) be a profinite group, let \((S_i)_{i\in I}\) be a family
		of pairwise nonisomorphic finite nonabelian simple groups, and suppose
		that, for every \(i\in I\), there is a continuous surjective homomorphism
		\[
		f_i:\Gamma\twoheadrightarrow S_i.
		\]
		Then the product homomorphism
		\[
		f=(f_i)_{i\in I}:\Gamma\longrightarrow\prod_{i\in I}S_i
		\]
		is surjective.
	\end{lemma}
	
	\begin{proof}
		We first show by induction that \(f\) maps surjectively onto every finite
		subproduct.  This is clear for a subproduct with one factor.  Suppose it
		has been proved for a finite subset
		\(J\subset I\), put \(A=\prod_{i\in J}S_i\), and take
		\(j\in I\setminus J\).  The image \(H\) of \(\Gamma\) in
		\(A\times S_j\) projects surjectively onto both factors.  Set
		\[
		N_A=\{a\in A:(a,1)\in H\},\qquad
		N_j=\{s\in S_j:(1,s)\in H\}.
		\]
		As the kernel of the projection \(H\to A\), the subgroup
		\(\{1\}\times N_j\) is normal in \(H\). Its image under the
		surjective projection \(H\to S_j\) is \(N_j\), so \(N_j\) is
		normal in \(S_j\).
		
		For \(a\in A\), choose \(s\in S_j\) with \((a,s)\in H\).
		The coset \(sN_j\) does not depend on this choice: if
		\((a,s')\in H\) as well, then \((1,s^{-1}s')\in H\), and hence
		\(sN_j=s'N_j\). Multiplication in \(H\) shows that the resulting map
		\[
		A\longrightarrow S_j/N_j,\qquad a\longmapsto sN_j
		\]
		is a homomorphism. It is surjective because \(H\to S_j\) is
		surjective. Its kernel is \(N_A\), since \(s\in N_j\) holds exactly
		when \((a,1)\in H\). The first isomorphism theorem therefore gives
		\[
		A/N_A\cong S_j/N_j.
		\]
		
		If \(H\ne A\times S_j\), then this common quotient is nontrivial.  Since
		\(S_j\) is simple, it follows that \(N_j=1\), and hence \(S_j\) is a
		quotient of \(A\).  On the other hand, every epimorphism
		\(A\twoheadrightarrow S_j\) restricts nontrivially to at least one
		factor \(S_i\).  The image of that factor is normal in \(S_j\), so the
		restriction is surjective; simplicity then makes it an isomorphism
		\(S_i\cong S_j\), contrary to the pairwise nonisomorphism assumption. Thus \(H=A\times S_j\), completing the induction.
		
		The preceding argument shows that \(f(\Gamma)\) projects surjectively
		onto every finite subproduct. A basic open set in
		\(\prod_{i\in I}S_i\) restricts only finitely many coordinates,
		so it meets \(f(\Gamma)\) whenever it is nonempty. Thus
		\(f(\Gamma)\) is dense.
		
		The map \(f\) is continuous because each \(f_i\) is continuous.
		Since \(\Gamma\) is compact, its image is compact and hence closed
		in the Hausdorff product \(\prod_{i\in I}S_i\).
		Being both dense and closed, \(f(\Gamma)\) is the whole product.
	\end{proof}
	
	\begin{proof}[Proof of Corollary~\ref{cor:infinite-product}]
		To obtain the simple quotients needed for
		Lemma~\ref{lem:independent-simple-quotients}, take the fixed curve
		\(X\) over \(\Fp\) and the representations \(\rho_r\), \(r\geq4\),
		constructed in the proof of Theorem~\ref{thm:main}, and put
		\(K=\Fp(X)\). Composing each \(\rho_r\) with the canonical quotient gives
		\[
		\overline\rho_r:\pi_1^{\et}(X)
		\xrightarrow{\rho_r}\SL_2(\mathbf F_{p^r})
		\twoheadrightarrow\PSL_2(\mathbf F_{p^r}).
		\]
		Since \(\rho_r\) is geometric with image \(\SL_2(\mathbf F_{p^r})\),
		the homomorphism \(\overline\rho_r\)
		and its restriction to \(\pi_1^{\et}(X_{\Fbar})\) are both surjective.
		For \(q>3\), the group \(\PSL_2(\mathbf F_q)\) is
		nonabelian simple, and
		\begin{equation}\label{eq:psl2-order}
			\#\PSL_2(\mathbf F_q)=\frac{q(q^2-1)}{\gcd(2,q-1)};
		\end{equation}
		see \cite[\S\S3.3.1--3.3.2, pp.~44--46]{Wilson09}.
		For fixed \(p\), the denominator is constant as \(q=p^r\) varies,
		and these orders strictly increase with \(r\). Thus the groups
		\(\PSL_2(\mathbf F_{p^r})\), \(r\geq4\), are pairwise nonisomorphic.
		Applying
		Lemma~\ref{lem:independent-simple-quotients} to \(\pi_1^{\et}(X)\) and to
		\(\pi_1^{\et}(X_{\Fbar})\), respectively, shows that the product homomorphism
		\[
		\overline\rho=(\overline\rho_r)_{r\geq4}:
		\pi_1^{\et}(X)\longrightarrow
		P:=\prod_{r\geq4}\PSL_2(\mathbf F_{p^r})
		\]
		satisfies
		\[
		\overline\rho\bigl(\pi_1^{\et}(X)\bigr)
		=
		\overline\rho\bigl(\pi_1^{\et}(X_{\Fbar})\bigr)
		=P.
		\]
		
		Compose \(\overline\rho\) with the canonical quotient
		\(G_K\twoheadrightarrow\pi_1^{\et}(X)\), and let \(L\) be the fixed field
		of its kernel. Then \(L/K\) is Galois and
		\(\Gal(L/K)\cong P\). Since the homomorphism factors through the
		\'{e}tale fundamental group of the projective curve \(X\), every finite
		subextension of \(L/K\) is unramified at every place.
		To rule out new constants, consider any finite Galois subextension
		of \(L/K\). Its Galois group is a finite quotient of \(P\), and
		\(\pi_1^{\et}(X_{\Fbar})\) still maps onto it because
		\(\overline\rho(\pi_1^{\et}(X_{\Fbar}))=P\).
		The discussion following \eqref{eq:fundamental-exact-sequence}
		therefore shows that this subextension has full constant field \(\Fp\).
		Every element of \(L\) lies in one of these finite Galois subextensions
		\cite[proof of Thm.~9.22.4]{Stacks}.
		Thus \(\Fp\) is algebraically closed in \(L\), proving the first assertion.
		
		For the second assertion, suppose to the contrary that such an extension
		\(L/K\) exists.
		From the factors \(\PSL_2(\mathbf F_q)\), we will construct infinitely
		many pairwise nonisomorphic continuous semisimple geometric
		representations of \(G_K\) over \(\Fbar\), all of the same dimension
		and everywhere unramified.
		We will use dimension \(4\) when \(p\) is odd and dimension \(2\)
		when \(p=2\).
		Since \(K\) has characteristic \(p'\ne p\), these representations,
		whose Artin conductors are zero, will contradict the
		cross-characteristic finiteness theorem
		\cite[Thm.~1.3, p.~2]{Luo26a}.
		
		After removing from \(S\) the finitely many values \(q\leq3\),
		the set \(S\) is still infinite.
		For each \(q\in S\), the group
		\(\PSL_2(\mathbf F_q)\) is a nonabelian simple quotient
		of \(\Gal(L/K)\).
		The order formula \eqref{eq:psl2-order} shows that their orders strictly increase
		with \(q\), so these groups are pairwise nonisomorphic.
		
		For each \(q\in S\), let
		\[
		\varphi_q:G_K\twoheadrightarrow\PSL_2(\mathbf F_q)
		\]
		be the homomorphism obtained by composing
		\(G_K\twoheadrightarrow\Gal(L/K)\) with the projection onto the
		\(q\)-th factor. Let \(L_q\) be the fixed field of
		\(\ker(\varphi_q)\). Then \(L_q/K\) is finite and unramified at every
		place. Since \(k_0\) is algebraically closed in \(L\), it is also
		algebraically closed in \(L_q\); hence \(L_q\) is regular over \(k_0\).
		
		Suppose first that \(p\) is odd. The natural two-dimensional
		representation of \(\SL_2(\mathbf F_q)\) does not descend to
		\(\PSL_2(\mathbf F_q)\), because \(-I_2\) acts nontrivially.
		Instead, let \(\SL_2(\mathbf F_q)\) act on \(\mathrm{M}_2(\Fbar)\)
		by conjugation,
		\[
		g\cdot B=gBg^{-1}.
		\]
		The center \(\{\pm I_2\}\) acts trivially, so this gives
		a four-dimensional representation
		\[
		a_q:\PSL_2(\mathbf F_q)\longrightarrow\GL_4(\Fbar).
		\]
		This representation is faithful: an element acting trivially by
		conjugation is represented by a matrix commuting with the full matrix
		algebra, hence by a scalar matrix, and is therefore trivial in \(\PSL_2(\mathbf F_q)\).
		
		To apply the finiteness theorem, let \(a_q^{\mathrm{ss}}\)
		be a semisimplification of \(a_q\).
		We show that it remains faithful.
		If it were not, its kernel would be a nontrivial normal subgroup
		of the simple group \(\PSL_2(\mathbf F_q)\), so \(a_q^{\mathrm{ss}}\) would be trivial.
		Every Jordan--H\"older factor of \(\mathrm{M}_2(\Fbar)\)
		would then be trivial.
		A basis adapted to a composition series would place \(a_q(\PSL_2(\mathbf F_q))\)
		inside an upper unitriangular group, which is solvable.
		This would make \(\PSL_2(\mathbf F_q)\) solvable, since \(a_q\) is faithful,
		contradicting the fact that \(\PSL_2(\mathbf F_q)\) is nonabelian simple.
		Thus \(a_q^{\mathrm{ss}}\) is faithful.
		
		Composing with \(\varphi_q\), we obtain
		\[
		\sigma_q:=a_q^{\mathrm{ss}}\circ\varphi_q:
		G_K\longrightarrow\GL_4(\Fbar).
		\]
		This representation is continuous because \(\varphi_q\) is
		a continuous homomorphism to a finite discrete group,
		and it is semisimple because \(\varphi_q\) is surjective.
		Since \(a_q^{\mathrm{ss}}\) is faithful,
		\(\ker(\sigma_q)=\ker(\varphi_q)\), so the fixed field
		of \(\ker(\sigma_q)\) is \(L_q\).
		We have already shown that \(L_q\) is regular over \(k_0\)
		and that \(L_q/K\) is everywhere unramified.
		Thus \(\sigma_q\) is geometric and everywhere unramified,
		with Artin conductor zero.
		
		The image of \(\sigma_q\) is isomorphic to
		\(\PSL_2(\mathbf F_q)\).
		These groups are pairwise nonisomorphic as \(q\) varies in \(S\),
		so the representations \(\sigma_q\) are pairwise nonisomorphic.
		We have therefore obtained infinitely many continuous semisimple
		geometric representations of \(G_K\), all of dimension \(4\)
		and with Artin conductor zero.
		This contradicts \cite[Thm.~1.3, p.~2]{Luo26a}.
		
		It remains to consider \(p=2\). In this case the center of
		\(\SL_2(\mathbf F_q)\) is trivial, so
		\[
		\PSL_2(\mathbf F_q)=\SL_2(\mathbf F_q).
		\]
		We can therefore use the natural two-dimensional representation
		\[
		b_q:\PSL_2(\mathbf F_q)\hookrightarrow\GL_2(\Fbar).
		\]
		It is faithful, and Lemma~\ref{lem:natural-irred} shows that
		it is absolutely irreducible. Thus no semisimplification is needed.
		
		Composing with \(\varphi_q\), we obtain
		\[
		\tau_q:=b_q\circ\varphi_q:
		G_K\longrightarrow\GL_2(\Fbar).
		\]
		These representations are continuous and absolutely irreducible,
		since \(\varphi_q\) is a continuous surjection onto the finite group
		\(\PSL_2(\mathbf F_q)\).
		The faithfulness of \(b_q\) gives
		\(\ker(\tau_q)=\ker(\varphi_q)\), so the fixed field
		of \(\ker(\tau_q)\) is again \(L_q\).
		As above, each \(\tau_q\) is geometric and everywhere unramified,
		with Artin conductor zero.
		Their images are isomorphic to the pairwise nonisomorphic groups
		\(\PSL_2(\mathbf F_q)\), so the representations \(\tau_q\),
		\(q\in S\), are pairwise nonisomorphic.
		This contradicts \cite[Thm.~1.3, p.~2]{Luo26a}, applied in
		dimension \(2\), and completes the proof.
	\end{proof}
	
	\begin{proof}[Proof of Corollary \ref{cor:nonliftable-family}]
		By Theorem~\ref{conditionaltheorem}, the image orders of the continuous semi\-simple geometric two-dimensional representations of
		\(\pi_1^{\et}(X)\) that admit char\-ac\-ter\-is\-tic-zero
		lifts are bounded by a constant \(B_X\).
		On the other hand, the order of \(\operatorname{Im}(\rho_r)\) tends to infinity with \(r\). Choosing \(r_0\geq 4\) such that this order exceeds \(B_X\) for every \(r\geq r_0\) proves the assertion.
	\end{proof}
	
	\section*{Declaration of generative AI use}
	Before seeking AI assistance, the authors had proved
	Theorem~\ref{conditionaltheorem} and, while searching for representations
	that admit no characteristic-zero lift, had conjectured the conclusion
	of Theorem~\ref{thm:main}. They subsequently used ChatGPT (Sol 5.6)
	to explore this conjecture. Its initial contribution was limited to
	a weaker result in a special case with \(p=2\). The authors recognized
	that this construction could be generalized and developed the statement
	and proof of Theorem~\ref{thm:main} in their present form.
	ChatGPT was also used to assist in checking the mathematical arguments throughout the manuscript. All mathematical arguments have been independently verified and
	reproduced by the authors. The authors bear sole responsibility for the correctness
	and completeness of the mathematical content and for the accuracy
	and completeness of the references.
	
	\bibliographystyle{amsalpha}
	\bibliography{etale-finiteness-counterexamples-revised}

@misc{Stacks,
	author       = {The {Stacks Project Authors}},
	title        = {The {Stacks} Project},
	year         = {2026},
	howpublished = {\url{https://stacks.math.columbia.edu}},
	url          = {https://stacks.math.columbia.edu}
}

@book{Serre77,
	author    = {J.-P. Serre},
	title     = {Linear Representations of Finite Groups},
	series    = {Graduate Texts in Mathematics},
	volume    = {42},
	publisher = {Springer},
	address   = {New York},
	year      = {1977},
	doi       = {10.1007/978-1-4684-9458-7},
	note      = {\href{https://doi.org/10.1007/978-1-4684-9458-7}{doi:10.1007/978-1-4684-9458-7}}
}

@book{RibesZalesskii10,
	author    = {L. Ribes and P. Zalesskii},
	title     = {Profinite Groups},
	edition   = {2nd},
	series    = {Ergebnisse der Mathematik und ihrer Grenzgebiete. 3. Folge},
	volume    = {40},
	publisher = {Springer},
	address   = {Berlin, Heidelberg},
	year      = {2010},
	doi       = {10.1007/978-3-642-01642-4},
	note      = {\href{https://doi.org/10.1007/978-3-642-01642-4}{doi:10.1007/978-3-642-01642-4}}
}

@book{Wilson09,
	author    = {R. A. Wilson},
	title     = {The Finite Simple Groups},
	series    = {Graduate Texts in Mathematics},
	volume    = {251},
	publisher = {Springer},
	address   = {London},
	year      = {2009},
	doi       = {10.1007/978-1-84800-988-2},
	note      = {\href{https://doi.org/10.1007/978-1-84800-988-2}{doi:10.1007/978-1-84800-988-2}}
}

@book{Mumford08,
	author    = {D. Mumford},
	title     = {Abelian Varieties},
	edition   = {2nd},
	series    = {Tata Institute of Fundamental Research Studies in Mathematics},
	volume    = {5},
	publisher = {Hindustan Book Agency},
	address   = {New Delhi},
	year      = {2008},
	note      = {Corrected reprint},
	isbn      = {978-81-85931-86-9}
}

@article{Dri74,
	author  = {V. G. Drinfeld},
	title   = {Elliptic modules},
	journal = {Math. USSR-Sb.},
	volume  = {23},
	number  = {4},
	year    = {1974},
	pages   = {561--592},
	doi     = {10.1070/SM1974v023n04ABEH001731},
	note    = {\href{https://doi.org/10.1070/SM1974v023n04ABEH001731}{doi:10.1070/SM1974v023n04ABEH001731}}
}

@article{Har74,
	author  = {G. Harder},
	title   = {Chevalley groups over function fields and automorphic forms},
	journal = {Ann. of Math. (2)},
	volume  = {100},
	number  = {2},
	year    = {1974},
	pages   = {249--306},
	doi     = {10.2307/1971073},
	url     = {https://doi.org/10.2307/1971073},
	note    = {\href{https://doi.org/10.2307/1971073}{doi:10.2307/1971073}}
}

@book{Laf97,
	author    = {L. Lafforgue},
	title     = {Chtoucas de {Drinfeld} et conjecture de {Ramanujan--Petersson}},
	series    = {Ast{\'e}risque},
	number    = {243},
	publisher = {Soci{\'e}t{\'e} Math{\'e}matique de France},
	address   = {Paris},
	year      = {1997},
	pages     = {1--329},
	doi       = {10.24033/ast.383},
	url       = {https://numdam.org/item/AST_1997__243__1_0/},
	note      = {\href{https://doi.org/10.24033/ast.383}{doi:10.24033/ast.383}}
}

@article{LRS93,
	author  = {G. Laumon and M. Rapoport and U. Stuhler},
	title   = {{$D$}-elliptic sheaves and the {Langlands} correspondence},
	journal = {Invent. Math.},
	volume  = {113},
	number  = {2},
	year    = {1993},
	pages   = {217--338},
	doi     = {10.1007/BF01244308},
	note    = {\href{https://doi.org/10.1007/BF01244308}{doi:10.1007/BF01244308}}
}

@article{Luo26a,
	author  = {Y. Luo},
	title   = {A finiteness theorem for mod {$p$} {Galois} representations over global function fields},
	journal = {Ramanujan J.},
	volume  = {70},
	number  = {3},
	year    = {2026},
	month   = jul,
	pages   = {Art.~50},
	doi     = {10.1007/s11139-026-01427-0},
	url     = {https://doi.org/10.1007/s11139-026-01427-0},
	note    = {\href{https://doi.org/10.1007/s11139-026-01427-0}{doi:10.1007/s11139-026-01427-0}}
}

@article{MoonTaguchi01,
	author  = {H. Moon and Y. Taguchi},
	title   = {Mod {$p$} {Galois} representations of solvable image},
	journal = {Proc. Amer. Math. Soc.},
	volume  = {129},
	number  = {9},
	year    = {2001},
	pages   = {2529--2534},
	doi     = {10.1090/S0002-9939-01-05894-4},
	url     = {https://doi.org/10.1090/S0002-9939-01-05894-4},
	note    = {\href{https://doi.org/10.1090/S0002-9939-01-05894-4}{doi:10.1090/S0002-9939-01-05894-4}}
}

@article{Taguchi17,
	author  = {Y. Taguchi},
	title   = {Moduli of {Galois} representations},
	journal = {Publ. Res. Inst. Math. Sci.},
	volume  = {53},
	number  = {4},
	year    = {2017},
	pages   = {457--516},
	doi     = {10.4171/PRIMS/53-4-1},
	url     = {https://doi.org/10.4171/PRIMS/53-4-1},
	note    = {\href{https://doi.org/10.4171/PRIMS/53-4-1}{doi:10.4171/PRIMS/53-4-1}}
}

@misc{Luo26b,
	author        = {Y. Luo},
	title         = {On the finiteness of geometric representations for varieties over finite fields},
	year          = {2026},
	eprint        = {2606.31341},
	archivePrefix = {arXiv},
	primaryClass  = {math.NT},
	doi           = {10.48550/arXiv.2606.31341},
	note          = {\href{https://arxiv.org/abs/2606.31341v2}{arXiv:2606.31341v2}}
}

@article{Pap09,
	author  = {M. Papikian},
	title   = {Modular varieties of {$D$}-elliptic sheaves and the {Weil--Deligne} bound},
	journal = {J. Reine Angew. Math.},
	volume  = {626},
	year    = {2009},
	pages   = {115--134},
	doi     = {10.1515/CRELLE.2009.004},
	eprint  = {0802.1568},
	archivePrefix = {arXiv},
	note    = {\href{https://doi.org/10.1515/CRELLE.2009.004}{doi:10.1515/CRELLE.2009.004}; \href{https://arxiv.org/abs/0802.1568}{arXiv:0802.1568}}
}

@book{Rei03,
	author    = {I. Reiner},
	title     = {Maximal Orders},
	series    = {London Mathematical Society Monographs. New Series},
	number    = {28},
	publisher = {Oxford University Press},
	address   = {Oxford},
	year      = {2003},
	note      = {Corrected reprint of the 1975 original; \href{https://doi.org/10.1093/oso/9780198526735.001.0001}{doi:10.1093/oso/9780198526735.001.0001}},
	isbn      = {978-0-19-852673-5},
	doi       = {10.1093/oso/9780198526735.001.0001}
}

@article{Abe18,
	author  = {Abe, T.},
	title   = {Langlands correspondence for isocrystals and the existence of crystalline companions for curves},
	journal = {J. Amer. Math. Soc.},
	volume  = {31},
	number  = {4},
	year    = {2018},
	pages   = {921--1057},
	doi     = {10.1090/jams/898},
	url     = {https://doi.org/10.1090/jams/898},
	note    = {\href{https://doi.org/10.1090/jams/898}{doi:10.1090/jams/898}},
}

@article{Kedlaya22,
	author  = {Kedlaya, K. S.},
	title   = {{\'E}tale and crystalline companions, {I}},
	journal = {\'Epijournal G\'eom. Alg\'ebrique},
	volume  = {6},
	year    = {2022},
	pages   = {Art.~20, 30~pp.},
	doi     = {10.46298/epiga.2022.6820},
	url     = {https://doi.org/10.46298/epiga.2022.6820},
	note    = {\href{https://doi.org/10.46298/epiga.2022.6820}{doi:10.46298/epiga.2022.6820}},
}

@InCollection{FontaineMazur97,
	author     = {Fontaine, J.-M. and Mazur, B.},
	booktitle  = {Elliptic curves, modular forms, \& {F}ermat's last theorem ({H}ong {K}ong, 1993)},
	editor     = {Coates, J. H. and Yau, S.-T.},
	edition    = {2nd},
	publisher  = {International Press},
	title      = {Geometric {G}alois representations},
	year       = {1997},
	pages      = {190--227},
}

@Article{MR2551763,
	author     = {Khare, C. and Wintenberger, J.-P.},
	journal    = {Invent. Math.},
	title      = {Serre's modularity conjecture. {I}},
	year       = {2009},
	issn       = {0020-9910,1432-1297},
	number     = {3},
	pages      = {485--504},
	volume     = {178},
	doi        = {10.1007/s00222-009-0205-7},
	fjournal   = {Inventiones Mathematicae},
	mrclass    = {11F80 (11F11 11F33 11R39)},
	mrnumber   = {2551763},
	mrreviewer = {Gabor\ Wiese},
	url        = {https://doi.org/10.1007/s00222-009-0205-7},
	note       = {\href{https://doi.org/10.1007/s00222-009-0205-7}{doi:10.1007/s00222-009-0205-7}},
}

@Article{MR2551764,
	author     = {Khare, C. and Wintenberger, J.-P.},
	journal    = {Invent. Math.},
	title      = {Serre's modularity conjecture. {II}},
	year       = {2009},
	issn       = {0020-9910,1432-1297},
	number     = {3},
	pages      = {505--586},
	volume     = {178},
	doi        = {10.1007/s00222-009-0206-6},
	fjournal   = {Inventiones Mathematicae},
	mrclass    = {11F80 (11F11 11F33 11R39)},
	mrnumber   = {2551764},
	mrreviewer = {Gabor\ Wiese},
	url        = {https://doi.org/10.1007/s00222-009-0206-6},
	note       = {\href{https://doi.org/10.1007/s00222-009-0206-6}{doi:10.1007/s00222-009-0206-6}},
}

@Article{Kisin09,
	author     = {Kisin, M.},
	journal    = {Invent. Math.},
	title      = {Modularity of {$2$}-adic {Barsotti--Tate} representations},
	year       = {2009},
	number     = {3},
	pages      = {587--634},
	volume     = {178},
	doi        = {10.1007/s00222-009-0207-5},
	url        = {https://doi.org/10.1007/s00222-009-0207-5},
	note       = {\href{https://doi.org/10.1007/s00222-009-0207-5}{doi:10.1007/s00222-009-0207-5}},
}

@Article{MR1751924,
	author     = {Khare, C.},
	journal    = {J. Ramanujan Math. Soc.},
	title      = {Conjectures on finiteness of mod {$p$} {G}alois representations},
	year       = {2000},
	issn       = {0970-1249,2320-3110},
	number     = {1},
	pages      = {23--42},
	volume     = {15},
	fjournal   = {Journal of the Ramanujan Mathematical Society},
	mrclass    = {11F80 (11R32)},
	mrnumber   = {1751924},
	mrreviewer = {Gebhard\ B\"ockle},
}

@Article{MR1782427,
	author     = {Moon, H.},
	journal    = {J. Number Theory},
	title      = {Finiteness results on certain mod {$p$} {G}alois representations},
	year       = {2000},
	issn       = {0022-314X,1096-1658},
	number     = {1},
	pages      = {156--165},
	volume     = {84},
	doi        = {10.1006/jnth.2000.2534},
	fjournal   = {Journal of Number Theory},
	mrclass    = {11F80 (11R32)},
	mrnumber   = {1782427},
	mrreviewer = {Gebhard\ B\"ockle},
	url        = {https://doi.org/10.1006/jnth.2000.2534},
	note       = {\href{https://doi.org/10.1006/jnth.2000.2534}{doi:10.1006/jnth.2000.2534}},
}

@article {MR2730374,
	AUTHOR = {Buzzard, K. and Diamond, F. and Jarvis, F.},
	TITLE = {On {S}erre's conjecture for mod {$\ell$} {G}alois
	representations over totally real fields},
	JOURNAL = {Duke Math. J.},
	FJOURNAL = {Duke Mathematical Journal},
	VOLUME = {155},
	YEAR = {2010},
	NUMBER = {1},
	PAGES = {105--161},
	ISSN = {0012-7094,1547-7398},
	MRCLASS = {11F80 (11F33 11F41)},
	MRNUMBER = {2730374},
	MRREVIEWER = {Michael\ M.\ Schein},
	DOI = {10.1215/00127094-2010-052},
	URL = {https://doi.org/10.1215/00127094-2010-052},
	NOTE = {\href{https://doi.org/10.1215/00127094-2010-052}{doi:10.1215/00127094-2010-052}},
}

@article {MR885783,
	AUTHOR = {Serre, J.-P.},
	TITLE = {Sur les repr\'esentations modulaires de degr\'e{} {$2$} de
	{$\operatorname{Gal}(\overline{\mathbf{Q}}/\mathbf{Q})$}},
	JOURNAL = {Duke Math. J.},
	FJOURNAL = {Duke Mathematical Journal},
	VOLUME = {54},
	YEAR = {1987},
	NUMBER = {1},
	PAGES = {179--230},
	ISSN = {0012-7094,1547-7398},
	MRCLASS = {11F11 (11G05 14G15 14G25 14K15)},
	MRNUMBER = {885783},
	MRREVIEWER = {M.\ A.\ Kenku},
	DOI = {10.1215/S0012-7094-87-05413-5},
	URL = {https://doi.org/10.1215/S0012-7094-87-05413-5},
	NOTE = {\href{https://doi.org/10.1215/S0012-7094-87-05413-5}{doi:10.1215/S0012-7094-87-05413-5}},
}

@article{deJong01,
	author  = {A. J. de Jong},
	title   = {A conjecture on arithmetic fundamental groups},
	journal = {Israel J. Math.},
	volume  = {121},
	year    = {2001},
	pages   = {61--84},
	doi     = {10.1007/BF02802496},
	note    = {\href{https://doi.org/10.1007/BF02802496}
	{doi:10.1007/BF02802496}}
}

@article{Gaitsgory07,
	author  = {D. Gaitsgory},
	title   = {On {de Jong}'s conjecture},
	journal = {Israel J. Math.},
	volume  = {157},
	year    = {2007},
	pages   = {155--191},
	doi     = {10.1007/s11856-006-0006-2},
	eprint  = {math/0402184},
	archivePrefix = {arXiv},
	note    = {\href{https://doi.org/10.1007/s11856-006-0006-2}
	{doi:10.1007/s11856-006-0006-2};
	\href{https://arxiv.org/abs/math/0402184}
	{arXiv:math/0402184}}
}

@incollection{FKV99,
	author    = {G. Frey and E. Kani and H. V{\"o}lklein},
	title     = {Curves with infinite {$K$}-rational geometric fundamental group},
	editor    = {H. V{\"o}lklein and J. G. Thompson and
	D. Harbater and P. M{\"u}ller},
	booktitle = {Aspects of {Galois} Theory ({Gainesville}, {FL}, 1996)},
	series    = {London Mathematical Society Lecture Note Series},
	volume    = {256},
	publisher = {Cambridge University Press},
	address   = {Cambridge},
	year      = {1999},
	pages     = {85--118}
}
	
\end{document}